\documentclass[11pt, a4paper]{amsart}

\usepackage[utf8]{inputenc}
\usepackage{fontenc}
\usepackage{amsmath, amssymb, amsthm, amsfonts}
\usepackage{mathrsfs}      
\usepackage{esint}         
\usepackage{geometry}      
\usepackage{enumitem}      
\usepackage{mathtools}
\usepackage{stmaryrd}
\usepackage{enumitem}

\usepackage{lineno}

\usepackage[usenames,dvipsnames]{xcolor}

\newtheorem{theorem}{Theorem}[section]
\newtheorem{lemma}[theorem]{Lemma}
\newtheorem{proposition}[theorem]{Proposition}
\newtheorem{corollary}[theorem]{Corollary}

\theoremstyle{definition}
\newtheorem{definition}[theorem]{Definition}

\newtheorem{remark}[theorem]{Remark}

\usepackage[color=Aquamarine,textsize=footnotesize]{todonotes}
\newcommand{\mc}{\mathcal}

\newcommand{\Vol}{\operatorname{Vol}}

\newcommand{\til}{\widetilde}

\DeclareMathOperator{\Jac}{Jac}

\DeclareMathOperator{\RCD}{\ensuremath{\mathsf{RCD}}}
\DeclareMathOperator{\CAT}{\ensuremath{\mathsf{CAT}}}

\DeclareMathOperator{\bary}{bar}

\DeclareMathOperator{\argmin}{argmin}

\newcommand{\R}{{\bf R}}
\newcommand{\Z}{{\bf Z}}

\newcommand{\Ga}{\Gamma}
\newcommand{\ga}{\gamma}

\newcommand{\cout}[1]{}

\renewcommand{\hat}{\widehat}
\renewcommand{\bar}{\overline}

\def\be#1\ee{\begin{align}\begin{split} #1 \end{split}\end{align}}
\def\beq#1\eeq{\begin{align*}\begin{split} #1 \end{split}\end{align*}}

\definecolor{lime}{HTML}{A6CE39}

\newcommand{\Hn}{\mathcal{H}^N}             

\newcommand{\SigmaP}{\Sigma_p X}             
\newcommand{\Kp}{K_p X}                      
\newcommand{\Xreg}{X_{{\rm reg}}}                  
\newcommand{\Xsing}{S_X}                     

\newcommand{\cur}{\mathbf{I}}

\usepackage{hyperref}

\hypersetup{
  colorlinks   = true,
  urlcolor     = teal,
  linkcolor    = teal,
  citecolor   = teal
}

\usepackage[capitalize]{cleveref}

\begin{document}

\title[The entropy-degree theorem for Alexandrov spaces]{The entropy-degree theorem \\ for Alexandrov spaces}

\author[P. Su\'arez-Serrato]{P. Su\'arez-Serrato} 
\date{\today}
\address{Max-Planck Institut f\"ur Mathematik, Vivatsgasse 7, Bonn, Deutschland}
\address{University of California, Santa Barbara, Isla Vista, California}
\address{{\it On leave from}: Instituto de Matem\'aticas, Universidad Nacional Aut\'onoma de M\'exico UNAM, Mexico Tenochtitlan}

%

\begin{abstract}
We present the entropy-degree theorem for Lipschitz maps between Alexandrov spaces with curvature bounded below, extending the classical Besson--Courtois--Gallot entropy-rigidity results to this singular setting. 
The proof requires a new degree theorem for Alexandrov spaces, developed using the Ambrosio--Kirchheim theory of integral currents, showing the equivalence between analytical and topological degrees. 

Applications include geometric obstructions for negatively curved Einstein metrics on 4-orbifolds, volume bounds for cone-manifolds, quantitative inequalities for hyperbolic convex cores, and lower bounds on the asymptotic translation lengths of end-periodic surface homeomorphisms. 
We show that entropy-volume minimization under uniform lower curvature bounds obstructs to the formation of metric singularities in Gromov--Hausdorff limits, prove an Alexandrov boundary rigidity theorem, and establish volume minima for cone manifolds and cone orbifolds.
\end{abstract}

\maketitle

\tableofcontents

\section{Introduction}

Let $(Z,d,m)$ be a metric measure space.
Define the {\em volume entropy} $h(Z)$ as (cf \cite{M, BCG-Samb, Reviron}),
\[
h(Z)=\limsup_{R\to\infty}\frac{\log m (B(x,R))}{R},
\]
where $B(x,R)$ is the geodesic ball of radius $R$ in $Z$ centered at $x$ in $Z$ (this asymptotic growth rate is independent of the base point $x$).
Usually the volume entropy is defined with respect to the universal cover. 
Here, the traditional volume entropy is denoted by $h(\til{Z})$, whereas $h(Z)$ denotes the exponential volume growth rate on the space $Z$ itself. 
This difference is relevant to our arguments, as we will exploit the entropy of intermediate covering spaces, building upon a strategy pioneered by Sambusetti \cite{Sambusetti-1999}.
Observe that whenever $Z$ is a length space with a cocompact group of isometries, the $\limsup$ in the definition of $h(Z)$ converges to a true limit (both in smooth manifolds \cite{M}, and in metric measure spaces \cite[Proposition 3.3]{BCG-Samb}).

An \textit{Alexandrov space} $X$ of curvature bounded below by $\kappa$,  is a complete, locally compact length metric space satisfying a triangle comparison property formally analogous to Toponogov's theorem.
 We assume throughout that $X$ is an $N$-dimensional, closed (compact and boundaryless) Alexandrov space, denoting by $\mc{H}^{N}$ the Hausdorff measure of dimension $N$.
Alexandrov spaces may have singularities, and this flexibility makes them better suited in certain geometric constructions than smooth manifolds.
In \Cref{sec:Alex-geom} we include the definition of an Alexandrov space as well as the properties we will use.

Given a continuous map $f:X\to Y$ between topological spaces, consider $f_*:\pi_1(X)\to\pi_1(Y)$ the induced map on fundamental groups, and let $\bar{X}$ be the cover of $X$ of the corresponding subgroup $\ker f_*<\pi_1(X)$.
 Then  $\pi_1(\bar{X})=\ker f_*$ and $\Ga:=\pi_1(X)/\ker f_*$ acts on $\bar{X}$ by deck transformations. 
 Let $\til{Y}$ be the universal cover of $Y$.
 Notice that $\bar{X}$ is the smallest cover for which there exists a lift of $f$ to a map $\til{f}:\bar{X}\to\til{Y}$.

Our main result is:

\begin{theorem}\label{thm:Alex-entropy-degree}
Let  the natural number $N\geq 3$,  $(X,d,\mc{H}^{N})$ be an oriented $N$-Alexandrov  space without boundary, and $Y$ be a closed orientable negatively curved locally symmetric space of dimension $N$. 
Then, for any continuous map $f:X\to Y$,
\begin{equation}\label{eq:Alex-entropy-degree-ineq}
h(\overline{X})^N \mathcal{H}^N(X) \ge |\deg (f)| h(\til{Y})^N \mathcal{H}^N(Y).
\end{equation}

Moreover, if  (\ref{eq:Alex-entropy-degree-ineq}) is an equality and $|\deg (f)| = [\pi_1(Y) : f_* \pi_1(X)]$, then the space $X$ is isometric to a locally symmetric smooth Riemannian manifold, and $f$ is homotopic to a Riemannian covering of degree $|\deg (f)|$ (up to homothetically rescaling the metric on $X$).
\end{theorem}

Douady and Earle \cite{Douady-Earle-1986} intiated the use of the conformal barycenter of measures to extend circle homeomorphisms to the disc.
The barycenter method was then developed further for smooth manifolds by Besson, Courtois, and Gallot \cite{Besson-Courtois-Gallot:95} to prove the entropy rigidity of smooth locally symmetric spaces. 
Adaptations to singular spaces were pioneered by Storm \cite{Storm-2002, Storm-2006, Storm-2007}. 
While Storm's approach constructed the BCG barycenter map using visual measures on horoboundaries and relied on Reshetnyak's theory of quasiregular mappings for geometric rigidity, here we see Alexandrov geometry within the context of non-collapsed metric measure spaces \cite{RCDbary}. 

We continue the barycenter method tradition here by now integrating the Ambrosio--Kirchheim theory of metric currents \cite{ambrosio2000currents}, which accommodates Lipschitz maps of arbitrary degree on Alexandrov spaces.
Namely, we show in \Cref{thm:an-deg-equals-top-deg} that the notions of topological and analytic degree coincide for Lipschitz maps on Alexandrov spaces. 
This equivalence was the missing piece of the puzzle in earlier work on the barycenter method on singular spaces.
It is also an aspect that had to be carefully substituted using {\it ad hoc} topological indices in our previous work about metric measure spaces \cite{RCDbary}.
We hope that this approach will allow the application of similar methods to other integral current spaces, beyond the Alexandrov realm. 

While building the proof of  \Cref{thm:Alex-entropy-degree} one technical point became apparent (which also limited the scope of our previous  work on metric measure spaces to the non-collapsed case).
When using metric Jacobians, it is important to rely on the Hausdorff measure so that the definitions of metric differentials by Kirchheim and Lytchak agree. 
This intricate detail is crucial for the inequalities in the barycenter method.
The proof of  \Cref{thm:Alex-entropy-degree} illuminates this constraint of the barycenter method, in that the reference measure of the metric measure spaces involved must be the Hausdorff measure (or perhaps a multiple of it).
In the case of Alexandrov spaces, this is not a restriction, as the Hausdorff measure is the natural measure to consider. 

Once these difficulties are overcome, the proof of \Cref{thm:Alex-entropy-degree} follows by methodically adapting our work on the barycenter method and volume entropy on metric measure spaces \cite{RCDbary}, itself resting on work by Bessi\'eres, Besson, Connell, Farb, Courtois, Gallot, and Sambusetti \cite{Bessieres-1998, Besson-Courtois-Gallot:95, BCG-Samb, Connell-Farb-2003, Sambusetti-1999}. 

\medskip

Observe that the rigidity statement of \Cref{thm:Alex-entropy-degree} intersects with, but fundamentally differs from, recent developments in Lipschitz-volume rigidity. 
Available rigidity results, such as those by Li \cite{li2015lipschitz} for Alexandrov spaces, {\it a priori} require the initial map to be $1$-Lipschitz and volume-preserving . 
Recent generalizations have relaxed the domain hypothesis, proving isometric rigidity for $1$-Lipschitz, degree-one maps from integral current spaces targeting Euclidean domains (Del Nin--Perales \cite{DelNin-Perales-2023}) or infinitesimally Euclidean manifolds (Z\"ust \cite{Zust-2024}), as well as maps from linearly locally contractible metric manifolds (Basso--Marti--Wenger \cite{Basso-Marti-Wenger-2025}). In contrast, the rigidity in \Cref{thm:Alex-entropy-degree} only assumes the existence of an arbitrary degree $k \ge 1$ continuous map into a negatively curved target. 
The barycenter method creates an isometry in the equality case in \Cref{thm:Alex-entropy-degree}, in the process constructing the intermediate $1$-Lipschitz, volume-preserving map. 

\medskip
We will next describe several applications of our main result, which extend prior work in diverse directions.
\medskip

In the past few years there has been significant interest in the existence of distinguished metrics on Gromov--Thurston manifolds, which are constructed as finite cyclic covers of hyperbolic manifolds branched along totally geodesic codimension-two submanifolds. 
Fine and Premoselli \cite{FP20}, and more recently Hamenst\"adt and J\"ackel \cite{HJ24}, constructed negatively curved Einstein metrics on such branched covers that do not admit hyperbolic metrics. 
Very recently, Hamenst\"adt \cite{Hamenstadt26} analyzed the geometry of these spaces, relying on metric measure approximations to establish strict volume inequalities in dimension 3. 
Using our \Cref{thm:Alex-entropy-degree}, we bypass these limit approximations, thus generalizing Hamenst\"adt's volume gap to all dimensions $N \ge 3$, while simultaneously providing a topological obstruction for the existence of Einstein metrics on 4-orbifolds.

Let $M$ be a closed hyperbolic $N$-manifold ($N \ge 3$), $\Sigma \subset M$ a totally geodesic codimension-two submanifold, and $M_d$ the degree-$d$ cyclic branched cover over $\Sigma$. 
The quotient $X = M_d / \mathbf{Z}_d$ naturally carries a hyperbolic cone metric with cone angle $2\pi/d \le \pi$, making it an $N$-dimensional Alexandrov space.

\begin{corollary}\label{cor:hamenstadt-gen}
For any dimension $N \ge 3$ and any degree $d \ge 2$, the $N$-dimensional Hausdorff measure of the singular quotient orbifold $X = M_d / \mathbf{Z}_d$ strictly bounds the volume of the smooth base manifold:
\[
\mathcal{H}^N(X) > \mathcal{H}^N(M).
\]
Consequently, any $\mathbf{Z}_d$-invariant normalized negatively curved Einstein metric on the branched cover $M_d$ satisfies $\operatorname{Vol}(M_d) > d \cdot \operatorname{Vol}(M)$.
\end{corollary}

In dimension four, this strict volume-degree inequality couples with the Chern--Gauss--Bonnet theorem to produce a topological obstruction for the existence of negatively curved orbifold Einstein metrics.

\begin{corollary}\label{cor:einstein-obstruction}
Let $X$ be a closed 4-dimensional Alexandrov orbifold with a non-empty singular set. If $X$ admits a negatively curved Einstein metric, then for any continuous map $f: X \to Y$ to a smooth hyperbolic 4-manifold $Y$, the orbifold Euler characteristic satisfies $\chi_{{\rm orb}}(X) \geq |\deg(f)| \chi(Y)$.
Moreover, if  $|\deg (f)| = [\pi_1(Y) : f_* \pi_1(X)]$ holds, then $\chi_{{\rm orb}}(X) > |\deg(f)| \chi(Y)$.

If the cyclic quotient $X = M_d / \mathbf{Z}_d$ of a Gromov--Thurston manifold admits a negatively curved Einstein metric, the Euler characteristic of its branch locus must be strictly negative ($\chi(\Sigma) < 0$).
\end{corollary}

Previous applications of the barycenter method to singular metric spaces were developed by Storm \cite{Storm-2002, Storm-2006, Storm-2007}, who established volume bounds for cone-manifolds, Alexandrov spaces, and hyperbolic convex cores. Because a global analytic degree theory for singular spaces was previously unavailable, those results were restricted to homotopy equivalences. 
With \Cref{thm:Alex-entropy-degree} we may now generalize these volume bounds to Lipschitz maps of arbitrary degree. 

The first extension applies to cone-manifolds.
\begin{corollary}\label{cor:cone-manifold}
Let $Z$ be a compact $N$-dimensional ($N \ge 3$) cone-manifold with cone angles $\le 2\pi$, and sectional curvatures $ \ge -1$ on its regular part. Let $Y$ be a closed orientable negatively curved locally symmetric space of dimension $N$. 
For any continuous map $f: Z \to Y$,
\[
(N-1)^N \mathcal{H}^N(Z) \ge |\deg(f)| h(\til{Y})^N \mathcal{H}^N(Y).
\]
\end{corollary}

The second extension generalizes the volume-entropy inequality to manifolds with boundary, removing previous homotopy equivalence requirements (cf \cite[Theorem 8.9]{Storm-2007}).

\begin{corollary}\label{cor:manifold-boundary}
Let $Z$ be a compact convex Riemannian $N$-manifold with boundary ($N \ge 3$) such that the sectional curvature of the interior of $Z$ is bounded below by $-1$. 
Let $Y_{{\rm geod}}$ be a compact convex hyperbolic $N$-manifold with totally geodesic boundary. 
For any proper continuous map $f: (Z, \partial Z) \to (Y_{{\rm geod}}, \partial Y_{{\rm geod}})$ of relative degree $k$,
\[
\operatorname{Vol}(Z) \ge |k| \operatorname{Vol}(Y_{{\rm geod}}).
\]
\end{corollary}

The third extension applies to the convex cores of hyperbolic 3-manifolds. 
Let $\mathbb{H}^3$ denote they real hyperbolic space of dimension 3.
Consider a non-uniform lattice $\Gamma$ acting by isometries on $\mathbb{H}^3$.
Recall that a hyperbolic manifold $M$ is called convex cocompact if it contains a compact convex subset $C_M$, called the convex core, such that the inclusion $C_M \hookrightarrow M$ is a homotopy equivalence. 
Two metrics are equivalent if there exists an isometry between them isotopic to the identity.
A complete hyperbolic 3-manifold $M = \mathbb{H}^3/\Gamma$ is called \emph{geometrically finite} if its convex core $C_M \subset M$ has finite volume, $\operatorname{Vol}(C_M) < \infty$. 
Storm \cite{Storm-2007} proved that if $M_g$ is a convex cocompact hyperbolic 3-manifold whose convex core $C_{M_g}$ has totally geodesic boundary, then $C_{M_g}$ minimizes volume among all hyperbolic manifolds in the same homotopy class. 
Having access to \Cref{thm:Alex-entropy-degree}, we now offer a degree-weighted generalization across distinct homotopy classes.

\begin{corollary}\label{cor:bonahon-degree}
Let $M$ be a geometrically finite hyperbolic 3-manifold with convex core $C_M$. 
Let $M_g$ be a convex cocompact hyperbolic 3-manifold such that the boundary of its convex core $\partial C_{M_g}$ is totally geodesic. 
If there exists a proper continuous map $f: (C_M, \partial C_M) \to (C_{M_g}, \partial C_{M_g})$ of relative degree $k$, then
\[
\operatorname{Vol}(C_M) \ge |k| \operatorname{Vol}(C_{M_g}).
\]
\end{corollary}


A hyperbolic 3-manifold $Z$ is \emph{acylindrical} if it contains no essential, non-peripheral, embedded incompressible annuli. 
Equivalently, every $\pi_1$-injective map from a cylinder $S^1 \times [0,1]$ into $Z$ is homotopic to a map into the boundary $\partial Z$.
Let $N$ be a compact $3$-manifold with boundary. 
Define the  space $CC(Z)$ as the set of equivalence classes of convex cocompact hyperbolic metrics on the interior of $Z$. 
Bridgeman, Brock, and Bromberg \cite{Bridgeman-2023} found a quantitative lower bound for the minimal volume of convex cores using the Weil--Petersson gradient flow of renormalized volume. 
Their geometric flow method evaluates structures within a fixed homotopy class. 
Because \Cref{thm:Alex-entropy-degree} applies to proper continuous maps of arbitrary degree we obtain an inequality including both the topological degree and the Weil--Petersson distance.

\begin{corollary}\label{cor:unified-volume}
Let $Z$ and $Z'$ be compact hyperbolizable acylindrical 3-manifolds without torus boundary components. 
Let $M_{{\rm geod}} \in CC(Z)$ and $M'_{{\rm geod}} \in CC(Z')$ be the unique convex cocompact hyperbolic 3-manifolds such that the boundaries of their convex cores $\partial C_{M_{{\rm geod}}}$ and $\partial C_{M'_{{\rm geod}}}$ are totally geodesic. 
Let $M' \in CC(Z')$ be a convex cocompact hyperbolic 3-manifold. 
If there exists a proper continuous map $$f: (C_{M'}, \partial C_{M'}) \to (C_{M_{{\rm geod}}}, \partial C_{M_{{\rm geod}}})$$ of relative degree $k$, then
\[
\operatorname{Vol}(C_{M'}) \ge |k| \operatorname{Vol}(C_{M_{{\rm geod}}}) + A(\partial Z') d_{WP}(\partial_c M', \partial_c M'_{{\rm geod}}) - \delta ,
\]
where $A(\partial Z')$ and $\delta$ are the constants from \cite[Theorem C]{Bridgeman-2023}.
\end{corollary}


An \emph{infinite type surface} $S$ is a connected, orientable surface with non-finitely generated fundamental group, typically arising from an infinite collection of handles or punctures.
 An \emph{end-periodic homeomorphism} $f: S \to S$ of such a surface is a map that is periodic on all but finitely many ends of $S$. 
 Consider the \emph{mapping class group} $\operatorname{Mod}(S)$, which is the group of isotopy classes of orientation-preserving homeomorphisms of $S$. 
The \emph{curve complex} $\mathcal{C}(S)$ is the simplicial complex whose vertices are isotopy classes of essential, non-peripheral simple closed curves on $S$, where edges connect pairs of curves that can be realized disjointly. 
Let $d_{\mathcal{C}}$ denote the combinatorial distance in the vertex graph of $\mathcal{C}(S)$ and $\alpha$ is a choice of a filling multicurve.
For an element $\phi \in \operatorname{Mod}(S)$, its \emph{asymptotic translation length} $\tau(\phi)$ is defined as the limit 
$$\tau(\phi):=\lim_{i \to \infty} \frac{d_{\mathcal{C}}(\alpha, \phi^i(\alpha))}{i}.$$
 
We will now describe an application to end-periodic homeomorphisms of infinite-type surfaces and translation lengths.
Field, Kim, Leininger, and Loving \cite{Field-2023} proved an upper bound on the infimal volume $\underline{\operatorname{Vol}}(\overline{M}_{\phi})$ of all convex hyperbolic metrics on the compactified mapping torus $\overline{M}_{\phi}$ of an irreducible end-periodic homeomorphism $f'$ of and infinite-type surface.
 They bounded this infimal volume by the asymptotic translation length $\tau(\phi)$ on the pants graph: $\underline{\operatorname{Vol}}(\overline{M}_{\phi}) \le V_{{\rm oct}} \tau(\phi)$, where $V_{{\rm oct}}$ is the volume of a regular ideal octahedron. 

When applied to the metric doublings of these mapping tori, \Cref{thm:Alex-entropy-degree} bounds this infimal volume from below, and we find a topological obstruction for the translation length on the pants graph involving the degree.

\begin{corollary}\label{cor:translation-length}
Let $f'$ be an irreducible end-periodic homeomorphism of an infinite-type surface, and let $\overline{M}_{\phi}$ be its compactified mapping torus. 
Let $Y$ be a compact convex hyperbolic $3$-manifold with totally geodesic boundary. 

If there exists a proper continuous map $g:(\overline{M}_{\phi}, \partial \overline{M}_{\phi}) \to (Y, \partial Y)$ of relative degree $k$, then the asymptotic translation length $\tau(\phi)$ of $\phi$ on the pants graph satisfies
\[
\tau(\phi) \ge \frac{|k|}{V_{{\rm oct}}} \operatorname{Vol}(Y).
\]
\end{corollary}


Recall that Gromov's simplicial volume $\|M\|$ of a compact, oriented $n$-manifold $M$ is the minimal $L^1$-norm of a fundamental cycle in the singular chain complex of $M$ with real coefficients \cite{Gromov-1982}. 
Formally, it is defined as,

$$\|M\| = \inf \left\{ \sum_{i=1}^k |a_i| : [M] = \sum_{i=1}^k a_i \sigma_i \right\}.$$

Here $[M] \in H_n(M; \mathbf{R})$ is the fundamental class of $M$, and the infimum is taken over all representations of $[M]$ as a linear combination of singular $n$-simplices $\sigma_i$ with real coefficients $a_i$. 
It measures the topological complexity of the manifold; for instance, Thurston showed that the simplicial volume of a closed hyperbolic $3$-manifold is proportional to its Riemannian volume \cite{Thurston-1980}.
This was extended to all dimensions by Gromov,  showing that $\|M\| = \operatorname{Vol}(M) / v_n$, where $v_n$ is the volume of an ideal regular hyperbolic $n$-simplex \cite{Gromov-1982}.
Storm \cite{Storm-2007-Duke} confirmed a conjecture by Bonahon, proving that the volume of a hyperbolic convex core is bounded below by half the simplicial volume of the doubled manifold, provided the manifolds are homotopy equivalent. 
His proof relied on approximating the metric doubling of the convex core by a sequence of smooth Riemannian manifolds to apply the volume-entropy inequality. 

Now we can apply \Cref{thm:Alex-entropy-degree} directly to the singular metric doubling of the convex core. 
This eliminates the need for smooth approximations and extends the minimal volume bound to proper Lipschitz maps of arbitrary relative degree.
We thus find a generalization of the original conjecture of Bonahon, now bounding the volume of convex cores over different topological spaces.

\begin{theorem}\label{thm:simplicial-volume}
Let $M$ be a compact orientable acylindrical $3$-manifold with incompressible boundary and no torus boundary components. 
Let $Z$ be a convex cocompact hyperbolic $3$-manifold with convex core $C_Z$. 
If there exists a proper continuous map $f: (C_Z, \partial C_Z) \to (M, \partial M)$ of relative degree $k$, then
\[
\operatorname{Vol}(C_Z) \ge \frac{|k|}{2} \| DM \|.
\]
\end{theorem}

The boundary rigidity problem asks whether a Riemannian metric is uniquely determined by its boundary distance function. 
For non-positively curved surfaces, Croke \cite{Croke-1990} and Otal \cite{Otal-1990} proved it using geodesic flow dynamics. 
In higher dimensions, Besson, Courtois, and Gallot \cite{Besson-Courtois-Gallot:95} applied the barycenter map to prove boundary rigidity for simple locally symmetric spaces. 
Extensions to other dynamical settings, such as magnetic flows, were developed by Herreros \cite{Herreros-2014}.

Standard boundary rigidity theorems require the manifold to be simple, ruling out the existence conjugate points. 
By applying \Cref{thm:Alex-entropy-degree} to metric doublings, we can now drop the simplicity hypothesis and prove a boundary volume rigidity theorem for Alexandrov spaces (without needing smooth geodesic flows).

\begin{theorem}\label{thm:boundary-rigidity}
Let $N \ge 3$ and $Y$ be a compact orientable hyperbolic $N$-manifold with totally geodesic boundary. 
Let $X$ be a compact orientable $N$-dimensional Alexandrov space with curvature bounded below by $-1$ and locally convex boundary $\partial X$. 
Assume there exists a proper continuous map $f: (X, \partial X) \to (Y, \partial Y)$ of relative degree $k$ such that the restriction $f|_{\partial X}: \partial X \to \partial Y$ is a local isometry, and such that the double map $Df$ of $f$ satisfies $|\deg(Df)| = [\pi_1(DY): (Df)_*\pi_1(DX)]$ . 
Then
\[
\operatorname{Vol}(X) \ge |k| \operatorname{Vol}(Y).
\]

If $\operatorname{Vol}(X) = |k| \operatorname{Vol}(Y)$, then the boundary $\partial X$ is totally geodesic and $X$ is isometric to a degree $|k|$ Riemannian covering of $Y$. 
In particular, if $|k|=1$ and $f|_{\partial X}$ is an isometry, then $X$ is isometric to $Y$.
\end{theorem}

This boundary volume rigidity provides a global counterpart to recent measure-theoretic approaches. 
Z\"ust \cite{Zust-2024} developed methods to handle Lipschitz-volume rigidity among integral currents that includes spaces with boundaries, assuming the mapping is $1$-Lipschitz and measure-preserving.
 By applying the entropy-degree bound to the metric doubling, \Cref{thm:boundary-rigidity} deduces both the totally geodesic geometry of the boundary and the interior isometry strictly from a proper continuous map of relative degree $k$.

\medskip


One of the central areas of interest in geometric analysis is the stability of metric inequalities under Gromov--Hausdorff convergence. 
Sequences of Riemannian manifolds satisfying uniform lower sectional curvature bounds may develop metric singularities in the limit, which is an Alexandrov space.
Applying \Cref{thm:Alex-entropy-degree} directly on the limit space, we can show the next stability theorem for a volume minimizing sequence converging to a real hyperbolic manifold. 
The rigidity of the limit space illustrates how volume minimization under a uniform lower sectional curvature bound obstructs metric singularities under GH--convergence.

\begin{corollary}\label{cor:gh-stability}
Let $N \ge 3$, $Y$ be a real hyperbolic $N$-manifold, and $(M_i, g_i)$ be a sequence of closed Riemannian $N$-manifolds with sectional curvature bounded below by $-1$. 
Assume that, for each $i$, there exists a continuous map $f_i : M_i \to Y$ of absolute degree $k \ge 1$  such that $k = [\pi_1(Y) : (f_i)_*\pi_1(M_i)]$.
 If $(M_i, g_i)$ GH-converges to a compact metric space $(X, d_X)$ and satisfies$$ \lim_{i\to\infty} \Vol(M_i, g_i) = k  \Vol(Y), $$
 then the limit space $(X, d_X)$ has no metric singularities.
  The space $X$ is a smooth manifold isometric to a Riemannian covering of $Y$ of absolute degree $k$. 
  In particular, when $k = 1$, the limit $X$ is isometric to $Y$.
  \end{corollary}

In the case of other locally symmetric spaces, the following stability result follows from \Cref{thm:Alex-entropy-degree}.
The entropy-volume hypothesis on the  limit is needed because $h(\widetilde{Y}) < N-1$ when we normalize so that the sectional curvature is bounded below by $-1$.


\begin{corollary}\label{cor:gh-stability2}
Let $N \ge 3$, and $Y$ be a closed negatively curved locally symmetric $N$-manifold endowed with its locally symmetric metric $g_0$. 
Let $(M_i, g_i)$ be a sequence of closed Riemannian $N$-manifolds with sectional curvature bounded below by some real number $ \kappa$. 
Assume that, for each $i$, there exists a continuous map $f_i: M_i \to Y$ of absolute degree $k \ge 1$ such that $k = [\pi_1(Y) : (f_i)_*\pi_1(M_i)]$.
If the sequence $(M_i, g_i)$ GH--converges to a non-collapsed limit compact metric space $(X, d_X)$ such that
\[
\lim_{i \to \infty} h(\widetilde{M}_i, g_i)^N \operatorname{Vol}(M_i, g_i) = k h(\widetilde{Y}, g_0)^N \operatorname{Vol}(Y, g_0),
\]
then the limit space $(X, d_X)$ has no metric singularities, and $X$ is a smooth manifold isometric to a degree $k$ Riemannian covering of $(Y, g_0)$. 
In particular, when $k=1$ the limit $X$ is  isometric to $(Y, g_0)$.
\end{corollary}

In comparison, Song showed a stability theorem for the volume entropy of hyperbolic manifolds without assuming \textit{a priori} curvature bounds \cite{Song-2025}. 
He proved that sequences of metrics approaching the minimal volume entropy converge to the hyperbolic metric in the measured GH--topology but only after the excision of subsets of vanishing volume.
Song also noted that strict GH--stability fails in this unconstrained setting, because sequences can develop thin, volume-negligible threads that diverge in the GH--topology. 

Corollaries \ref{cor:gh-stability} and \ref{cor:gh-stability2} show that a uniform lower curvature bound is a sufficient geometric obstruction to the formation of such threads. 
The Alexandrov curvature bound guarantees GH--precompactness, and a non-collapsed Alexandrov limit space. 
Analyzing this singular limit directly via \Cref{thm:Alex-entropy-degree}, we show GH--stability, without requiring subset excision, showing that entropy-volume minimization impedes the formation of metric singularities.

Moreover, the stability criteria in Corollaries \ref{cor:gh-stability} and \ref{cor:gh-stability2}  complement recent results on metric rigidity. 
Whereas Basso, Marti, and Wenger \cite{Basso-Marti-Wenger-2025} prove Lipschitz-volume rigidity for metric spaces that are homeomorphic to manifolds, their approach needs the limit space to already be homeomorphic to a manifold and a given $1$-Lipschitz map. 
In contrast, our entropy-volume minimization here actively converges to the GH--limit space  metric singularities, simultaneously finding the homeomorphims, proving the Riemannian smoothness, and the showing isometry starting only from a continuous map.

\medskip 

Yamaguchi produced an upper bound for the simplicial volume for Alexandrov spaces with curvature bounded below by $-1$ \cite[Theorem 0.5]{yamaguchi1997simplicial} .
This include hyperbolic cone-manifolds with cone angles $\le 2\pi$. 
With \Cref{thm:Alex-entropy-degree}, we can find sharp lower bounds for the volume of such Alexandrov spaces, in terms of the simplicial volume. 

\begin{corollary}\label{cor:cone-simplicial-volume}
Let $X$ be a closed $N$-dimensional Alexandrov space with curvature bounded below by $-1$. 
Let $Y$ be a closed smooth hyperbolic $N$-manifold, $f \colon X \to Y$ be a continuous map of absolute degree $k \ge 1$, and $v_N$ the maximal volume of an ideal regular simplex in $\mathbb{H}^N$. 
Then
\begin{equation}\label{eq:cone-vol-bound}
    \mathcal{H}^N(X) \ge k \operatorname{Vol}(Y) = k v_N \|Y\|,
\end{equation}

In particular, if $f$ is a homotopy equivalence, then $\|X\| = \|Y\|$, and 
\begin{equation}\label{eq:cone-simplicial-bound}
    \mathcal{H}^N(X) \ge v_N \|X\|.
\end{equation}
Equality in \eqref{eq:cone-simplicial-bound} occurs if and only if $X$ is isometric to $Y$. 

Consequently, if $X$ is a hyperbolic cone-manifold with a non-empty singular set (meaning there are cone angles strictly less than $2\pi$), the inequality is strict: $$\mathcal{H}^N(X) > v_N \|X\|.$$
\end{corollary}

This strict lower bound solves a specific difficulty in the deformation theory of hyperbolic cone-metrics. 
As detailed by Cooper, Hodgson, and Kerckhoff \cite[Theorem 3.20, Corollary 3.21]{cooper2000three}, the Schläfli formula determines that the volume of a hyperbolic cone-manifold strictly decreases as its cone angles continuously increase. 
This proves that the smooth hyperbolic structure (where all cone angles equal $2\pi$) is a local volume minimum along any continuous deformation path.

However, deducing that the smooth metric is an absolute volume minimum across the entire moduli space of cone-metrics is complicated by the topology of the deformation space. 
It remains an open question whether the space of hyperbolic Dehn surgery structures is connected  \cite[Section 5.6, Conjecture 2]{cooper2000three}. 
Be that as it may, using \Cref{thm:Alex-entropy-degree} we obtain a volume bound that is path-independent and that completely sidesteps the deformation space.

\begin{corollary}\label{cor:cone-manifold-volume}
Let $M$ be a closed $N$-manifold that admits a smooth hyperbolic metric $g_{\text{hyp}}$. 
Let $g_{\text{cone}}$ be a hyperbolic cone-metric on $M$ with a non-empty singular locus such that all cone angles satisfy $\alpha_i \le 2\pi$. 
Then the volume of the $g_{\text{cone}}$ is strictly larger thatn the volume of $g_{\text{hyp}}$:
\begin{equation}
    \mathcal{H}^N(M, g_{\text{cone}}) > \operatorname{Vol}(M, g_{\text{hyp}}).
\end{equation}
\end{corollary}

In one of the proofs of the orbifold theorem, controlling the volume of cone-metrics as they approach the orbifold angles is a critical analytic step. 
Boileau, Leeb, and Porti \cite[Section 6.2]{boileau2005geometrization} rely on the local differential properties of the Schläfli formula to track this volume variation. 
With \Cref{thm:Alex-entropy-degree}, we circumvent the need to analyse continuous deformation limits on orbifolds entirely, as well.
This allows us to find the volume minimum for orbifold cone-metrics.

\begin{corollary}\label{cor:orbifold-volume-minimality}
Let $\mathcal{O}$ be a closed $N$-dimensional orbifold modeled on a negatively curved locally symmetric space, and let $g_{\text{sym}}$ denote its canonical locally symmetric orbifold metric. Let $g_{\text{cone}}$ be a cone-metric on the underlying topological space $|\mathcal{O}|$ such that the cone angle along any singular stratum of ramification index $m_i$ satisfies $\alpha_i \le 2\pi/m_i$. 

The volume of the cone-metric is bounded from below by the volume of the locally symmetric metric,
\begin{equation}
    \mathcal{H}^N(|\mathcal{O}|, g_{\text{cone}}) \ge \operatorname{Vol}(\mathcal{O}, g_{\text{sym}}).
\end{equation}
Furthermore, if any cone angle satisfies the strict inequality $\alpha_i < 2\pi/m_i$, then the volume inequality is strict,
 $$\mathcal{H}^N(|\mathcal{O}|, g_{\text{cone}}) > \operatorname{Vol}(\mathcal{O}, g_{\text{sym}}).$$
\end{corollary}

Let us compare this result with Storm's volume rigidity theorem for Alexandrov spaces \cite{Storm-2006}. 
Storm's degree-$1$ theorem requires the target to be a closed, completely smooth real hyperbolic manifold and it cannot be applied directly to the orbifold $\mathcal{O}$.
In our approach we extended the the Besson-Courtois-Gallot barycenter method so that it can now cover these orbifold cases.
Moreover, Storm's theorem applies only to real hyperbolic geometry. 
\Cref{thm:Alex-entropy-degree} overcomes these limitations, operating across all negatively curved locally symmetric geometries while also tracking the ramification index.

\medskip

We will now explain why the hypothesis equating the degree to the fundamental index in \Cref{thm:Alex-entropy-degree} is needed.

\begin{remark}\label{rmk:rigidity_condition}
The condition $|\deg(f)| = [\pi_1(Y) : f_* \pi_1(X)]$ in the rigidity statement is equivalent to the lift of $f$ to the covering space corresponding to $f_* \pi_1(X)$ has absolute degree equal to $1$. 
Let $p: \hat{Y} \to Y$ be the covering space associated with the subgroup $f_* \pi_1(X) \subset \pi_1(Y)$. 
The degree of $p$ equals the index $[\pi_1(Y) : f_* \pi_1(X)]$. 
The map $f: X \to Y$ lifts to a map $\hat{f}: X \to \hat{Y}$ such that $f = p \circ \hat{f}$. 
By the multiplicativity of degrees,
\[
|\deg(f)| = |\deg(\hat{f})| \cdot |\deg(p)| = |\deg(\hat{f})| \cdot [\pi_1(Y) : f_* \pi_1(X)].
\]
Thus, the condition $|\deg(f)| = [\pi_1(Y) : f_* \pi_1(X)]$ holds if and only if $|\deg(\hat{f})| = 1$. 

When analyzing smooth manifolds (cf \cite{Besson-Courtois-Gallot:95}), this condition is not required as a hypothesis because the equality case makes the barycenter map into  a smooth local isometry, and a regular covering map. 
However, when considering Alexandrov spaces, the limit map is initially only known to be $1$-Lipschitz and volume-preserving almost everywhere. 
 The condition $|\deg(f)| = [\pi_1(Y) : f_*\pi_1(X)]$ implies the lift $\hat{f}$ to the intermediate cover has absolute degree exactly $1$. 
 This allows us to use the degree $1$ Lipschitz-volume rigidity theorem for Alexandrov spaces by Li and Wang  \cite{li2014lipschitz, li2015lipschitz}. 
 
The following construction (cf.\ \cite[Remark 3.2]{RCDbary}) provides an example of  where the equality $|\deg(f)| = [\pi_1(Y) : f_*\pi_1(X)]$ fails.
Let $N \ge 3$, $Y$ be a closed orientable hyperbolic $N$-manifold, and $X = Y \# Y$ be the connected sum of two copies of $Y$. 
The manifold $X$ admits a Riemannian metric, making it a compact $N$-dimensional Alexandrov space. 
Define a Lipschitz map $f: X \to Y$ by collapsing the connecting cylinder $S^{N-1} \times I$ to a point, reducing $X$ to the wedge sum $Y \vee Y$, and mapping both components identically to $Y$ via orientation-preserving maps.

Observe that the connecting sphere $S^{N-1}$ is simply connected, because $N \ge 3$. 
Moreover, $\pi_1(X) \cong \pi_1(Y) * \pi_1(Y)$ by the Seifert--van Kampen theorem. 
The induced homomorphism $f_*: \pi_1(X) \to \pi_1(Y)$ maps both free factors surjectively onto $\pi_1(Y)$. 
The image $f_*(\pi_1(X))$ equals $\pi_1(Y)$, so the index $[\pi_1(Y) : f_* \pi_1(X)] = 1$. 
The pushforward of the fundamental class $[X]$ is the sum of the pushforwards of the fundamental classes of the two components, which means $f_*([X]) = 2[Y]$. 
Therefore the absolute value of the topological degree is $|\deg(f)| = 2 > 1 =[\pi_1(Y) : f_* \pi_1(X)]$. 

If $X$ satisfied the entropy-volume equality, omitting this condition from the rigidity statement would imply that $X$ is globally isometric to $Y$.
This is a contradiction, as the connected sum $Y \# Y$ and the base manifold $Y$ have distinct Betti numbers.
\end{remark}

\subsection*{Organization of the paper}
The paper is organized as follows. 
\S 2 covers background in Alexandrov geometry, including structural properties, metric doublings, and also on homological invariants. 
\S 3 digests the Ambrosio--Kirchheim theory of integral currents in metric spaces, providing the support to circumvent the singular obstacles. 
\S 4 reviews metric differentiation and explains properties of the signed metric Jacobian.
\S 5 presents the statement and proof \Cref{thm:an-deg-equals-top-deg}, the key equivalence between the analytical and topological degrees. 
 \S 6 briefly describes the volume entropy bounds obtained via the barycenter method on metric measure spaces. 
 \S 7 combines all these ingredients to prove \Cref{thm:Alex-entropy-degree}, and deduce all the subsequent results.
 Finally, Alethia's self-report on the AI collaboration assisting this paper is found in the appendix.


\section{Alexandrov Geometry background}\label{sec:Alex-geom}

\subsection{Definition and Curvature Bounds}
An Alexandrov space is a complete, locally compact length space that generalizes Riemannian manifolds by satisfying a synthetic lower curvature bound.
The curvature bounds are defined via distance comparisons based on the Toponogov theorem, as follows
A length space $X$ has curvature bounded below by $\kappa$ in $\mathbf{R}$ if it satisfies that for any triangle $\Delta(p, q, r)$ in $X$ with perimeter less than $2\pi/\sqrt{\kappa}$ (for $\kappa > 0$), there exists a comparison triangle $\bar{\Delta}(\bar{p}, \bar{q}, \bar{r})$ in the simply connected space form $M_\kappa^2$ of constant curvature $\kappa$ such that $d(p, q) = d(\bar{p}, \bar{q})$, $d(q, r) = d(\bar{q}, \bar{r})$, and $d(r, p) = d(\bar{r}, \bar{p})$. 
For any point $x$ on a geodesic segment in $X$ between two vertices, and the corresponding point $\bar{x}$ on the comparison segment in $M_\kappa^2$, the Alexandrov condition requires
$d(z, x) \ge d(\bar{z}, \bar{x})$,
where $z$ is the vertex opposite the segment. 
This inequality characterizes the synthetic lower curvature bound by requiring that geodesic triangles in $X$ are at least as "thick" as their counterparts in $M_\kappa^2$.

A comprehensive introduction to Alexandrov geometry is provided by Burago--Burago--Ivanov \cite{buragobook}.
One of the important aspects of Alexandrov spaces  is that they are precisely the Gromov–Hausdorff limits of smooth Riemannian manifolds with a uniform lower sectional curvature bound. 
This is a cornerstone result of the structure theory developed by Burago, Gromov, and Perelman \cite{burago1992ad}, building on the foundational compactness results of Gromov \cite{Gromov-1981}.
This feature makes the study of Alexandrov geometry essential when understanding limits of geometric flows or convergence of sequences of manifolds, beyond its intrinsic interest.

\subsection{Local Structure}
A metric space $(X, d)$ is said to be \textit{locally asymptotically conical} at a point $p \in X$ if the rescaled pointed metric spaces $(X, \lambda d, p)$ converge in the pointed Gromov--Hausdorff sense to a metric cone as $\lambda \to \infty$.
At every point $p$ in $X$, the space is locally asymptotically conical. 
The \textit{space of directions} $\SigmaP$ is the completion of the set of geodesic directions at $p$. 
The \textit{tangent cone} $\Kp$, which arises exactly as this Gromov--Hausdorff limit, is the Euclidean cone $C(\SigmaP)$ over the space of directions, equipped with the distance:
\begin{equation}
    d_{\Kp}((t, \xi), (s, \eta))^2 = t^2 + s^2 - 2ts \cos d_{\SigmaP}(\xi, \eta).
\end{equation}

\subsection{Boundary}

Let $\mathbf{R}^{+}= [0,\infty)$, and recall that a point $p$ in $X$ is a \emph{boundary point} ($p$ in $\partial X$) if the tangent cone $T_p X$ is isometric to $\mathbf{R}^{N-1} \times \mathbf{R}^{+}$.
 Equivalently, $p$ lies in $\partial X$ if the space of directions $\Sigma_p X$ has non-empty boundary, in which case $\partial (\Sigma_p X) \cong \Sigma_p (\partial X)$.

\subsection{Rectifiability}\label{sec:N-rect-Alex}

A metric space is countably $N$-rectifiable if it can be covered, up to an $\mathcal{H}^N$-measure zero set, by a countable union of bi-Lipschitz images of subsets of ${\bf R}^N$. 
Otsu and Shioya \cite{otsu1994riemannian} proved that the singular set of an $N$-dimensional Alexandrov space has a Hausdorff dimension of at most $N-1$ (so its $\mathcal{H}^N$-measure equals zero). 
Furthermore, they proved the regular set has full measure and can be covered by a countable number of domains that are bi-Lipschitz equivalent to open subsets of  ${\bf R}^N$.

By the foundational structure theorems for Alexandrov spaces with curvature bounded below, the regular set of $X$ admits local bi-Lipschitz distance charts into $\mathbf{R}^N$ \cite{burago1992ad}, while the singular set's Hausdorff dimension at most $N-1$ \cite{otsu1994riemannian}.
 Consequently, $X$ is countably $N$-rectifiable with respect to the $N$-dimensional Hausdorff measure $\mathcal{H}^N$

\subsection{Singular and Regular Sets}

The singular structure of Alexandrov spaces is stratified. 
We distinguish between regular points, where the space resembles a smooth manifold, and singular points, where it does not.

A point $p$ in $X$ is \textit{regular} if $\Kp$ is isometric to $\R^N$. 
The set of regular points is denoted $\Xreg$. 
A point is \textit{singular} if it is not regular; the singular set is denoted $\Xsing$.

While it is possible to refine the description of the stratification detailing the dimension of the Euclidean factor in the tangent cone, for the purpose of the degree towards our goals, the binary distinction  regular versus singular will be enough.
The foundational topological stratification and the bound on the Hausdorff dimension of the singular set were originally established by Burago, Gromov, and Perelman \cite{burago1992ad}. 
Building upon this, Otsu and Shioya \cite{otsu1994riemannian} formalized the analytic and Riemannian properties of the regular set, demonstrating that the manifold domain is open, dense, and of full measure. 
A comprehensive modern synthesis of these structural results is provided in the textbook by Burago, Burago, and Ivanov \cite{buragobook}. 
We summarize the information we need in the following result.

\begin{theorem}[\cite{burago1992ad, otsu1994riemannian, buragobook}]\label{thm:AlexStructure}
Let $X$ be  an $N$-dimensional Alexandrov space, then,
\begin{enumerate}
    \item $\Xreg$ is dense in $X$;
    \item $\Xsing$ has Hausdorff dimension at most $N-1$;
    \item The $N$-dimensional Hausdorff measure of the singular set is zero, $\Hn(\Xsing) = 0$.
\end{enumerate}
\end{theorem}

\subsection{Locally Lipschitz Contractibility}

To relate metric analysis to topology, specifically to guarantee the existence of a fundamental class in singular homology, the space must be locally contractible in a controlled manner.

A metric space $X$ is said to be {\emph locally Lipschitz contractible} if for every point $p$ in $X$ and sufficiently small $r > 0$, the ball $B(p, r)$ can be contracted to a point inside $B(p, C r)$ via a Lipschitz homotopy.
Mitsuishi and Yamaguchi  proved that finite-dimensional Alexandrov spaces are locally lipschitz contractible \cite{mitsuishi2014llc}.

\begin{theorem}[Mitsuishi--Yamaguchi \cite{mitsuishi2014llc}]\label{thm:MY-Alex-sllc}
Let $X$ be an Alexandrov space of dimension $N$ and curvature bounded below by $\kappa$. 
Then $X$ is locally Lipschitz contractible. 
There exists a constant $C$ such that for any $x \in X$ and small $r$, there is a Lipschitz homotopy contracting $B(x, r)$ to $x$ within $B(x, Cr)$.
\end{theorem}

This property also implies that continuous maps can be uniformly approximated by Lipschitz maps, and that local homology groups behave essentially like those of manifolds.


\subsection{Lipschitz approximation of continuous maps}\label{sec:Lip-approx}
The Ambrosio--Kirchheim \cite{ambrosio2000currents} theory of metric currents requires Lipschitz maps to define metric differentials and pushforward currents. 
To apply degree-theoretic and homological arguments to continuous maps, we will now explain Lipschitz approximations within continuous homotopy classes. 
Let $X$ and $Y$ be compact finite-dimensional Alexandrov spaces with curvature bounded below. 
A consequence of  \Cref{thm:MY-Alex-sllc}, continuous maps between $X$ and $Y$ can be locally modified via Lipschitz homotopies. 
Therefore, a continuous map $f: X \to Y$ can be uniformly approximated by, and is therefore homotopic to, a Lipschitz map $f': X \to Y$.
 Hence we may assume, without loss of generality that any continuous map between such spaces may been replaced by a homotopic Lipschitz representative.

\subsection{Orientability}

For any point $x$ in $X$, the local $N$-homology group at $x$ is by definition:
$$H_N(X, X \setminus \{x\}; \mathbf{Z}) \cong H_{N-1}(\Sigma_x X; \mathbf{Z})$$
The isomorphism may be seen using the conic structure of neighborhoods.
If $x \in \Xreg $, then $\Sigma_x X \cong \mathbf{S}^{n-1}$, so $H_N(X, X \setminus \{x\}) \cong \mathbf{Z}$.
If $x \in S_X$, the structure of $\Sigma_x X$ is more complex. 
However, if $X$ is an $n$-dimensional Alexandrov space, $\Sigma_x X$ is an $(n-1)$-dimensional Alexandrov space. 

\begin{definition}[Mitsuishi  \cite{mitsuishi2016orientability}]\label{def:Alex-orient}
An $N$-Alexandrov space $X$ is orientable if the local homology is constant, i.e., there exists a fundamental class $[X] \in H_{N}(X; \Z)$ that spans the local homology at every point.
Meaning that that for every $x$ in $X$, the natural map $H_{N}(X) \to H_{N}(X, X \setminus \{x\})$ sends $[X]$ to a generator of the local group.
\end{definition}

Orientability in Alexandrov spaces has been explored through various definitions, for example by Harvey and Searle \cite{harvey2017orientation}.

\subsection{Universal Covering of Alexandrov Spaces}\label{sec:Alex-universal}

Consider an Alexandrov space $(X, d)$ with curvature bounded below by $\kappa$. 
The structural theory of Burago, Gromov, and Perelman \cite{burago1992ad} guarantees that every point $x$ in $X$ admits a neighborhood that is locally homeomorphic to the tangent cone $K_x X$. 
As metric cones are contractible, it follows that $X$ is locally contractible and, consequently, semi-locally simply connected. 
Thus, by standard covering space theory, $X$ admits a simply connected universal covering space $\til{X}$ equipped with a covering map $\pi \colon \til{X} \to X$.

We can endow $\til{X}$ with a pull-back length metric $\til{d}$ such that $\pi$ becomes a local isometry. 
The definition of an Alexandrov space having curvature bounded below by $\kappa$ is inherently local, relying on Toponogov triangle comparisons in sufficiently small neighborhoods. 
Because $\pi$ is a local isometry, $( \til{X}, \tilde{d} \,) $ inherits this $\kappa$ lower curvature bound. 
Therefore, the universal cover $(\til{X}, \tilde{d} \, )$ is also an Alexandrov space, with the same curvature bound $\kappa$.

\subsection{Lipschitz Volume Rigidity of Alexandrov spaces}

The following theorem translates optimal measure equality into isometric equivalence.
It is known as Lipschitz volume rigidity, and was shown for Alexandrov spaces by  Li--Wang \cite{li2014lipschitz} and Li \cite[Theorem A, Corollary 0.1]{li2015lipschitz}. 

\begin{theorem}[{Lipschitz measure rigidity, \cite{li2014lipschitz, li2015lipschitz}}]\label{thm:Lip-Vol-Rigidity}
Let $X$ and $Y$ be closed $N$-dimensional Alexandrov spaces. If there exists a surjective $1$-Lipschitz map $g \colon X \to Y$ such that
   $ \mathcal{H}^N(X) = \mathcal{H}^N(f(X))$,
then $X$ is isometric to $Y$, and the map $g$ is a metric isometry.
\end{theorem}

\subsection{Doubling, Gluing, and Cone-Manifolds}

By Perelman's doubling theorem \cite[Theorem 5.2]{Perelman91}, if $Z$ is a compact Alexandrov space with curvature bounded below by $\kappa$ and locally convex boundary $\partial Z$, its metric double $DZ = Z \cup_{\partial Z} Z$ is an Alexandrov space without boundary with curvature bounded below by $\kappa$. 
The requirement of local convexity is essential, preventing geodesics that take shortcuts across the boundary interface, which would otherwise result in thin triangles disobeying the synthetic lower curvature bound. 
This construction allows boundary-value problems to be analysed using closed spaces.
This result applies directly to convex Riemannian manifolds and can be extended to $\epsilon$-neighborhoods of convex cores, shown by Storm \cite[Theorem 5.3]{Storm-2002}. 

Metric doubling of these neighborhoods, strictly convex, $C^{1,1}$-regular, gives a controlled mechanism to transition from infinite-volume geometrically finite hyperbolic manifolds to closed, curvature-bounded Alexandrov spaces.

More generally, by Petrunin's gluing theorem \cite{Petrunin-1997}, gluing two Alexandrov spaces with curvature $\ge \kappa$ along isometric, locally convex boundaries preserves the lower curvature bound. 
This theorem establishes a robust cut-and-paste method for synthetic lower curvature bounds, generalizing Perelman's doubling to asymmetric amalgamations.

Finally, by the structure theory of Burago, Gromov, and Perelman \cite[p. 7]{burago1992ad}, a cone-manifold with cone angles $\le 2\pi$ and regular-part sectional curvatures $\ge \kappa$ is an Alexandrov space with curvature bounded below by $\kappa$. 
This result bridges piecewise-Riemannian geometry with Alexandrov structure theory. 
The restriction to cone angles $\le 2\pi$  prevents the concentration of negative curvature on the singular strata, thus making sure that Toponogov's comparison inequalities still hold globally.

\subsection{Orbifolds, Convex Cores, and Mapping Tori}
Let $(X, g_{E})$ be a closed $4$-dimensional Riemannian orbifold with an Einstein metric $g_{E}$ scaled so $\operatorname{Ric}(g_E) = -3g_E$.
Let $\chi_{{\rm orb}}(X)$ denote the orbifold Euler characteristic of $X$.
Then, because the squared norm of the Weyl tensor is non-negative, Satake's  \cite{Satake57} version of the Chern--Gauss--Bonnet theorem implies $$\operatorname{Vol}(X, g_E) \le \frac{4\pi^2}{3}\chi_{{\rm orb}}(X).$$
 Assume $X$ is a degree-$d$ cyclic branched cover $X = M_d / \mathbf{Z}_d$ over a codimension-two submanifold $\Sigma \subset M$.
 Then the Riemann--Hurwitz formula for branched covers (see Thurston \cite[Chapter 13]{Thurston-1980}) yields 
 $$\chi_{{\rm orb}}(X) = \chi(M) - \left(1 - \frac{1}{d}\right) \chi(\Sigma).$$

In $3$-dimensional hyperbolic geometry, the volume of convex cores obeys specific bounds. 
Storm \cite[Theorem 2.6]{Storm-2007-Duke} showed that the volume of the metric double $DC_{M_{{\rm geod}}}$ of the convex core of a convex cocompact hyperbolic $3$-manifold $M_{{\rm geod}}$ with totally geodesic boundary equals the simplicial volume of its topological double: 
$$\operatorname{Vol}(DC_{M_{\rm geod}}) = \|DM\|$$

\medskip

We briefly recall relevant geometric structures on Teichm\"uller space. 
For a given closed surface $S$, the \emph{Weil--Petersson distance}, denoted $d_{WP}$, is the Riemannian distance function on the Teichm\"uller space $\mathcal{T}(S)$ induced by the Weil--Petersson metric. 
The cotangent space at any point $X \in \mathcal{T}(S)$ is canonically identified with the space of holomorphic quadratic differentials on $X$, and the Weil--Petersson metric is defined via the $L^2$-inner product of these differentials with respect to the unique hyperbolic metric in the conformal class.

In the setting of a convex cocompact hyperbolic $3$-manifold $M'$, the renormalized volume $V_R$ acts as a smooth functional defined on the deformation space of $M'$, which is parameterized by the Teichm\"uller space of its conformal boundary, $\mathcal{T}(\partial_c M')$. 
The \emph{Weil--Petersson gradient flow} is the continuous dynamical system on $\mathcal{T}(\partial_c M')$ driven by the gradient of this renormalized volume functional with respect to the Weil--Petersson metric.

Bridgeman, Brock, and Bromberg \cite[Theorem C]{Bridgeman-2023} provided lower bounds for convex cores using the Weil--Petersson distance: 
\begin{equation}\label{eq:WP-convex-core-bound}
    \operatorname{Vol}(C_{M'}) - \operatorname{Vol}(C_{M'_{geod}}) \ge A(\partial M') \left( d_{WP}(\partial_c M', \partial_c M'_{geod})  \right) - \delta,
\end{equation}
where $A(\partial N') > 0$ is a constant depending exclusively on the topology of the boundary $\partial M'$, and $\delta > 0$ is a universal constant. 
These parameters, $A(\partial M') > 0$ and $\delta$, are proven to exist via analytic bounds on the Weil--Petersson gradient flow \cite{Bridgeman-2023}.
 
 \medskip
  
 Let $V_{{\rm oct}}$ be the volume of a regular ideal octahedron.
Consider an end-periodic homeomorphism of an infinite-type surface $f \colon S \to S$.
We denote by $\tau(f)$ the asymptotic translation length of $f$. This quantity measures the translation distance of $f$ acting on the pants graph $\mathcal{P}(S)$ of the surface. Because $S$ is of infinite type, $\mathcal{P}(S)$ is disconnected. However, the end-periodic assumption guarantees the existence of a non-empty subgraph $\mathcal{P}_f(S) \subset \mathcal{P}(S)$ whose connected components $\Omega$ are invariant under a positive power of $f$. For a vertex $x \in \Omega$, the translation length on that specific component is defined by the limit inferior
\[
    \tau(f, \Omega) = \liminf_{n \to \infty} \frac{d_{\mathcal{P}}(x, f^n(x))}{n},
\]
where $d_{\mathcal{P}}$ is the path metric on $\Omega$. The overall asymptotic translation length $\tau(f)$ is then defined as the infimum of $\tau(f, \Omega)$ over all such invariant components in $\mathcal{P}_f(S)$.

Field, Kim, Leininger, and Loving \cite[Theorem 1.1]{Field-2023} found the following bound on the infimal volume $\underline{\operatorname{Vol}}(\overline{M}_{f})$ of convex hyperbolic metrics on mapping tori $M_{f}$ of $f$, in terms of $V_{{\rm oct}}$ and $\tau(f)$: 
\begin{equation}\label{eq:MapTori-bound}
    \underline{\operatorname{Vol}}(\overline{M}_{f}) \le V_{{\rm oct}}\tau(f).
\end{equation}

\subsection{Homological Asymptotic Invariants and Pseudomanifolds}\label{sec:hom-inv}
Let us briefly recall why closed $N$-dimensional Alexandrov spaces are topological pseudomanifolds. 

Let $X$ be an $N$-dimensional Alexandrov space with curvature bounded below. We denote by $\mathcal{S}_X$ the singular set of $X$, which consists of all points in $X$ whose tangent cone is not isometric to the Euclidean space $\mathbf{R}^N$.

According to the structural theorems of Burago, Gromov, and Perelman \cite{burago1992ad}, the Hausdorff dimension of the singular set satisfies $\dim_{\mathcal{H}}(\mathcal{S}_X) \le N-1$. 
The $(N-1)$-dimensional stratum of the singular set actually corresponds exactly to the boundary of the space, $\partial X$. 
Therefore, when $X$ is a closed Alexandrov space (compact and with empty boundary),  the Hausdorff dimension of the singular set $\mathcal{S}_X$ satisfies,
\begin{equation}\label{eq:hausdorff-N-2}
    \dim_{\mathcal{H}}(\mathcal{S}_X) \le N-2.
\end{equation}

In spite of the their metric singularities, Alexandrov spaces enjoy tractable topological features. 
One such topological regularity is expressed by Perelman's triangulation theorem, bridging into combinatorial topology.

\begin{theorem}[\cite{Perelman91}]\label{thm:perelman-triangulation}
Let $X$ be an $N$-dimensional Alexandrov space with curvature bounded from below. 
Then $X$ is homeomorphic to a locally finite simplicial complex. In particular, if $X$ is a compact Alexandrov space, then $X$ is homeomorphic to a finite simplicial complex.
\end{theorem}

In a closed and orientable Alexandrov spaces, because the singular locus has Hausdorff dimension at most $N-2$, and the regular locus is dense and connected, the complex is strongly connected and every $(N-1)$-simplex is a face of exactly two $N$-simplices.
Furthermore, orientable Alexandrov spaces are integral homology manifolds \cite{mitsuishi2016orientability}. 
Consequently, they are closed oriented pseudomanifolds with a fundamental class $[X]$ in $H_N(X; \mathbf{Z})$.

\subsection{Gromov--Hausdorff Convergence and Stability}

By Gromov's compactness theorem \cite{Gromov-1981, burago1992ad}, a sequence of closed Riemannian $N$-manifolds $(M_i, g_i)$ with uniform lower curvature $K(g_i) \ge \kappa$ and upper diameter bounds admits a convergent subsequence  in the Gromov--Hausdorff topology. 
Such a sequence is called \emph{non-collapsed} if there exists a positive uniform lower bound on the volumes, $\operatorname{Vol}(M_i, g_i) \ge v > 0$. 
Under this non-collapsing condition, the GH--limit $X$ is a compact $N$-dimensional Alexandrov space with curvature bounded below by $\kappa$.
 Furthermore, the top-dimensional Hausdorff measure behaves continuously with respect to the GH--convergence \cite{burago1992ad}. 
 This continuity means that the volume does not decrease nor concentrates in the limit.
 Thus $\operatorname{Vol}(M_i, g_i)$  converges to the $N$--dimensional Hausdorff measure of the singular limit space,
\begin{equation}
    \lim_{i \to \infty} \operatorname{Vol}(M_i, g_i) = \mathcal{H}^N(X).
\end{equation}
This measure-theoretic continuity operates locally too, so that the volumes of corresponding metric balls converge to the Hausdorff measure of the limiting metric balls.
Topologically, such a non-collapsed limit is remarkably rigid, the $M_i$ are eventually--- for a large enough $i$---homeomorphic to the non-collapsed GH--limit. 
This follows from Perelman's stability theorem \cite{Perelman91, Kapovitch-2007}.

\begin{theorem}[Perelman's stability theorem {\cite[Theorem 1.1]{Kapovitch-2007}}]\label{thm:perelman-stability}
Let $X$ be a compact $N$-dimensional Alexandrov space with curvature $\ge \kappa$. 
Then there exists a constant $\epsilon = \epsilon(X) > 0$ such that for any $N$-dimensional Alexandrov space $Y$ with curvature $\ge \kappa$ satisfying $d_{GH}(X, Y) < \epsilon$, $Y$ is homeomorphic to $X$.
\end{theorem}

\section{Integral Currents}

We will now briefly review the theory of currents in metric spaces, as developed by Ambrosio and Kirchheim \cite{ambrosio2000currents}. 
Their ideas generalize de Rham currents to spaces without smooth forms.   
An $N$-dimensional metric current $T$ on $X$ is a multilinear functional on $(N+1)$-tuples of Lipschitz functions $(f, \pi_1, \dots, \pi_N)$, where $f$ is bounded and $\pi_i$ are Lipschitz.

Moreover, currents satisfy the following key axioms:
\begin{enumerate}
\item Finite Mass: $T$ defines a finite Radon measure $\llbracket T \rrbracket$. 
\item Locality: $T(f, \pi) = 0$ if $\pi_i$ is constant on the support of $f$.
\item Boundary: $\partial T$ is an $(N-1)$-current defined by $$\partial T(f, \pi_1, \dots, \pi_{N-1}) = T(1, f, \pi_1, \dots, \pi_{N-1}).$$
\end{enumerate}

The Ambrosio--Kirchheim theory replaces the space of smooth $k$-forms with a space of $(k+1)$-tuples of Lipschitz functions, written as $(f, \pi_1, \dots, \pi_k)$. 
These Lipschitz tuples formally represent the expression $f \, d\pi_1 \wedge \dots \wedge d\pi_k$ and act as \emph{metric differential forms}.
Ambrosio and Kirchheim \cite{ambrosio2000currents} then define a metric $k$-current $T$ as a real-valued multilinear functional on this space of Lipschitz tuples that is linear in each argument and satisfies specific continuity, locality, and finite mass axioms mentioned above.

We shall now recall the necessary metric current terminology, specifically the definitions of the boundary, integer rectifiable currents, and the weight of a current, following Ambrosio and Kirchheim \cite{jaramillo2021alexandrov}. 
Let $Z$ be a complete metric space, and let $\mathcal{D}^m(Z)$ denote the space of $(m+1)$-tuples $(f, \pi_1, \dots, \pi_m)$ of Lipschitz functions on $Z$, where $f$ is bounded. 
We denote the space of $m$-dimensional metric currents on $Z$ by $M_m(Z)$.

\begin{definition}[Boundary of a current, {\cite[Definition 2.25]{jaramillo2021alexandrov}}]
Let $T \in M_m(Z)$. The \textit{boundary} of $T$, denoted by $\partial T$, is the function $\partial T: \mathcal{D}^{m-1}(Z) \to \mathbf{R}$ given by
\[
    \partial T(f, \pi_1, \dots, \pi_{m-1}) = T(1, f, \pi_1, \dots, \pi_{m-1}).
\]
\end{definition}

\begin{definition}[Integer rectifiable current, {\cite[Definition 2.29]{jaramillo2021alexandrov}}]
Let $T \in M_m(Z)$. We say that $T$ is an \textit{$m$-dimensional integer rectifiable current} if it has a current parametrization, consisting of parametrizations and weight functions, $(\{\varphi_i\}, \{\theta_i\})$, satisfying the following conditions:
\begin{enumerate}
    \item The set of parametrizations $\varphi_i: A_i \subset \mathbf{R}^m \to Z$ is a countable collection of bi-Lipschitz maps such that all $A_i$ are precompact Borel measurable sets with pairwise disjoint images.
    \item The weight functions, $\theta_i \in L^1(A_i, \mathbf{N})$, are defined so that the following equalities hold:
    \[ T = \sum_{i=1}^\infty \varphi_{i\#} \llbracket \theta_i \rrbracket \quad \text{and} \quad M(T) = \sum_{i=1}^\infty M(\varphi_{i\#} \llbracket \theta_i \rrbracket). \]
\end{enumerate}
\end{definition}

\begin{definition}[Weight of a current, {\cite[Equation (2.2)]{jaramillo2021alexandrov}}]
Given an $m$-dimensional integer rectifiable current $T$ with parametrization $(\{\varphi_i\}, \{\theta_i\})$, let $\mathbf{1}_{\varphi_i(A_i)}$ be the indicator function of the image $\varphi_i(A_i)$. 
The \textit{weight} of $T$ is defined as the $L^1$ function $\theta_T: Z \to \mathbf{N} \cup \{0\}$ given by
\[
    \theta_T = \sum_{i=1}^\infty \theta_i \circ \varphi_i^{-1} \mathbf{1}_{\varphi_i(A_i)},
\]
\end{definition}

With these definitions in mind, we can state the structural results connecting the geometric current constructed from the regular set of an Alexandrov space to its homological generator.
The following theorems of Jaramillo--Perales--Rajan--Searle--Siffert situate currents as a fundamental object in the study of Alexandrov spaces. 

\begin{theorem}[Jaramillo--Perales--Rajan--Searle--Siffert {\cite[Theorem 4.6]{jaramillo2021alexandrov}}]\label{thm:jprss-4.6}
Let $(X, d)$ be an $n$-dimensional closed Alexandrov space, $T$ the $n$-current defined in \cite[Theorem 4.1]{jaramillo2021alexandrov} and $T'$ the $n$-current that generates the group $\{S \in I_n(X) \mid \partial S = 0\} \cong \Z$ from \cite[Theorem 2.37]{jaramillo2021alexandrov}. Then, either $T = T'$ or $T = -T'$. Hence, $\partial T = 0$.
\end{theorem}

\begin{theorem}[Jaramillo--Perales--Rajan--Searle--Siffert \cite{jaramillo2021alexandrov}]\label{thm:AlexIntCurrent} Let $X$ be an $N$-dimensional, closed, orientable Alexandrov space. 
Then $X$ defines an unique integer rectifiable current $T_X$ such that:
\begin{enumerate}
    \item The support of $T_X$ is $X$.
    \item The weight of $T_X$ equals $1$ for $\Hn$-almost all $x$.
    \item The boundary of $T_X$ is zero: $\partial T_X = 0$.
\end{enumerate}
\end{theorem}

Explicitly, for a tuple of Lipschitz functions $(f, \pi_1, \dots, \pi_N)$ on an Alexandrov space:
\begin{equation}
    T_X(f, \pi) = \int_{\Xreg} f(x) \det(d_x \pi) \, d\Hn(x).
\end{equation}
Here, $d_x \pi$ is the differential of the map $\pi = (\pi_1, \dots, \pi_n): X \to {\bf R}^N$ at the regular point $x$, where the tangent cone is identified with ${\bf R}^N$ via a chosen orientation.
The weight is exactly 1 because the Hausdorff measure on $\Xreg$ is normalized to coincide with the Euclidean Lebesgue measure in the tangent charts.

To relate the properties of metric currents to algebraic topology, we organize the groups of integral currents into a chain complex. Following Ambrosio and Kirchheim (see  \cite[Section 3.3]{mitsuishi2019homologies} and \cite[Section 2.4]{jaramillo2021alexandrov}.
 Let $I_k^c(X)$ denote the abelian group of $k$-dimensional integral currents on a metric space $X$ with compact support. 

By the boundary rectifiability theorem, the boundary of an integral current is also an integral current. 
So the boundary operator $\partial$ restricts to a well-defined group homomorphism $\partial_k: I_k^c(X) \to I_{k-1}^c(X)$. 
The boundary operator satisfies $\partial_{k-1} \circ \partial_k = 0$,  and hence the graded group $$I_\bullet^c(X) = \bigoplus_{k=0}^\infty I_k^c(X).$$ equipped with these boundary maps forms a chain complex.

\begin{definition}[Current homology, {\cite{mitsuishi2019homologies, jaramillo2021alexandrov}}]
The \textit{current homology} of a metric space $X$ in degree $k$, denoted $H_k^{\mathbf{I}}(X)$, is defined as the $k$-th homology group of the chain complex of integral currents with compact support $(I_\bullet^c(X), \partial)$:
\[
    H_k^{\mathbf{I}}(X) = \frac{\ker\left(\partial_k: I_k^c(X) \to I_{k-1}^c(X)\right)}{\operatorname{im}\left(\partial_{k+1}: I_{k+1}^c(X) \to I_k^c(X)\right)}.
\]
\end{definition}

When the space $X$ is compact, every current automatically has compact support. 
In this case, the chain group $I_k^c(X)$ coincides exactly with the group of all integral currents $I_k(X)$
Notably, for spaces that are locally Lipschitz contractible, the homology group of the chain complex of integral currents, $H_N^{\mathbf{I}}(X)$, is isomorphic to the singular homology $H_N(X; \Z)$ \cite[Theorem 1.3]{mitsuishi2019homologies}. 
Mitsuishi and Yamaguchi \cite{mitsuishi2014llc} showed that Alexandrov spaces satisfy this condition. 
To define this isomorphism explicitly (without cosheaves), we use the complex of singular Lipschitz chains $S_\bullet^{\operatorname{Lip}}(X)$. 
As $X$ is locally Lipschitz contractible, the natural inclusion $S_\bullet^{\operatorname{Lip}}(X) \hookrightarrow S_\bullet(X)$ induces an isomorphism on homology, $H_N^{\operatorname{Lip}}(X) \cong H_N(X; \Z)$. 
A Lipschitz $k$-simplex $\sigma: \Delta^k \to X$ then defines an integral current, by pushing forward the standard Euclidean current $\llbracket \Delta^k \rrbracket$ of the simplex, $ [\sigma] = \sigma_\# \llbracket \Delta^k \rrbracket$.
Extending this linearly produces a chain map $[\cdot]: S_\bullet^{\operatorname{Lip}}(X) \to I_\bullet^c(X)$. 
By Mitsuishi  \cite[Theorem 1.3]{mitsuishi2019homologies}, the induced map on homology $[\cdot]_*: H_N^{\operatorname{Lip}}(X) \to H_N^{\mathbf{I}}(X)$ is an isomorphism. 
Denote the inverse of this composed isomorphism by $\Psi: H_N^{\mathbf{I}}(X) \to H_N(X; \Z)$.
Under this isomorphism, the current $T_X$ maps to the fundamental class $[X]$. We formulate this fact in the following lemma.

\begin{lemma} \label{lem:fundamental-class-current}
Let $X$ be a closed, oriented $N$-dimensional Alexandrov space, and let $T_X$ in $ H_N^{\mathbf{I}}(X)$ be the fundamental current of $X$.
Under the isomorphism $\Psi$, the homology class of the current $T_X$ maps to the fundamental class $[X]$.
\end{lemma}

\begin{proof}
Because $X$ is an $N$-dimensional compact space, there are no non-trivial integral $(N+1)$-currents. 
As a result, the homology group $H_N^{\cur}(X)$ consists exactly of the integral $N$-currents satisfying $\partial T = 0$.

The topological orientation of $X$ determines a unique fundamental class $[X]$ in $H_N(X; \Z)$, which can be represented by a Lipschitz cycle $c = \sum_{i} a_i \sigma_i \in S_N^{\operatorname{Lip}}(X)$. 
The chain map sends $c$ to the current $T(c) = \sum_{i} a_i (\sigma_i)_\# \llbracket \Delta^N \rrbracket$. Since $\partial c = 0$ and the boundary operator commutes with the push-forward, we have $\partial T(c) = 0$. Furthermore, $[T(c)]$ acts as the generator for the current homology group $H_N^{\cur}(X) \cong \Z$.

Moreover, the geometric current $T_X$ constructed in \cite[Theorem 4.1]{jaramillo2021alexandrov} is an integer rectifiable current of weight $1$.
 In \cite[Theorem 4.6]{jaramillo2021alexandrov}, it is proved that $\partial T_X = 0$, and that any closed integral $N$-current of weight $1$ coincides with the homological generator $T(c)$, up to sign. 
 By selecting the orientation of the regular set $\Xreg$ to be compatible with the topological orientation of $X$, we obtain the exact equality $T_X = T(c)$ as currents. 
 Applying the inverse isomorphism $\Psi$ produces $\Psi([T_X]) = \Psi([T(c)]) = [X]$, completing the proof.
\end{proof}

Having painted the context of integral currents, we may now state the following result by Mitsuishi \cite[Theorem 1.8]{mitsuishi2016orientability} showing how several standard notions of orientability are equivalent in Alexandrov spaces.

\begin{theorem}[\cite{mitsuishi2016orientability}]
 For a closed $N$-Alexandrov space $X$, the following are equivalent:
\begin{enumerate}
\item $X$ is orientable (according to Definition \ref{def:Alex-orient}) .
\item $H_{N}(X; \Z) \cong \Z$.
\item $\Xreg$ is orientable as a Riemannian manifold, and the orientation extends across $S_X$.
\end{enumerate}
\end{theorem}

The equivalence relies on the interplay between the topological structure of Alexandrov spaces and the geometric measure theory of metric currents, underpinned by the homological isomorphism established by Mitsuishi \cite[Theorem 1.3]{mitsuishi2019homologies}.
For completeness, and for the reader's convenience, we will include a proof here.

\begin{proof}
\noindent \textbf{(1) $\iff$ (2):}
This equivalence is a topological consequence of $X$ behaving homologically like a manifold on a dense open set.
 In Definition \ref{def:Alex-orient}, orientability is formulated via the triviality of the local orientation cover $x \mapsto H_N(X, X \setminus \{x\}; \Z)$.
  On a connected space this implies the existence of a global fundamental class that spans the top homology group $H_N(X; \Z) \cong \Z$.

\noindent \textbf{(2) $\iff$ (3):}
Because $X$ is a finite-dimensional Alexandrov space, it is locally Lipschitz contractible \cite{mitsuishi2014llc}. 
By Mitsuishi \cite[Theorem 1.3]{mitsuishi2019homologies}, there is a canonical isomorphism between the standard integral singular homology and the homology of integral currents with compact support,
\[
H_N(X; \Z) \cong H_N^{IC}(X).
\]
Because $X$ is a closed space, all currents have compact support. 
Since $\dim(X) = N$, there are no non-trivial $(N+1)$-currents (i.e., $I_{N+1}(X) = 0$). 
Thus, $H_N^{IC}(X)$ consists exactly of the integral $N$-currents satisfying $\partial T = 0$.

By Theorem \ref{thm:AlexStructure}, the regular set $\Xreg$ is an open Riemannian manifold of full measure, since $\Hn(S_X) = 0$. 
The space $X$ is a closed Alexandrov space, so it has no boundary ($\partial X = \emptyset$). 
This implies that the singular set $S_X$ actually has Hausdorff dimension at most $N-2$. 
Removing a closed set of codimension $\ge 2$ cannot disconnect the space, guaranteeing that $\Xreg$ is connected.

Assume (2), so $H_N(X; \Z) \cong \Z$. By the isomorphism above, $H_N^{IC}(X) \cong \Z$, meaning there exists a non-trivial integral cycle $T_X$ in $ I_N(X)$ with $\partial T = 0$ spanning the group. 
Since $\Hn(S_X) = 0$, the mass measure of $T_X$ is entirely concentrated on $\Xreg$. 
As an integral current, $T_X$ is represented by a measurable orientation on $\Xreg$ and an integer-valued multiplicity function $\theta$ in $L^1(\Xreg; \Z)$. 
For the boundary $\partial T_X$ to vanish globally, the multiplicity $\theta$ must be locally constant almost everywhere on $\Xreg$. 
Connectivity then forces $\theta$ to be a non-zero global constant (which we may normalize to $1$). 
This constant multiplicity establishes that $\Xreg$ is a continuously orientable Riemannian manifold. 
Furthermore, the condition $\partial T_X = 0$ as a current on $X$ means that the boundary of this orientation current has zero mass on $S_X$.
That is, the orientation extends across $S_X$ in the sense of currents, yielding (3).

Conversely, assume (3). We define an integer rectifiable $N$-current $T_X = \llbracket \Xreg \rrbracket$ by endowing $\Xreg$ with its orientation and assigning it a constant multiplicity of $1$. 
Since $\Hn(S_X) = 0$ and $X$ is compact, $T_X$ has finite mass and is a well-defined integral $N$-current on $X$. 
The assumption that the orientation extends across $S_X$ means that there is no boundary mass concentrated on $S_X$.
Hence  $\partial T_X = 0$ in $I_{N-1}(X)$. 
Because $\Xreg$ is connected and has full measure, any other closed $N$-current is just an integer multiple of $T_X$. 
Thus, $T_X$ spans $H_N^{IC}(X)$, and we obtain $H_N^{IC}(X) \cong \Z$. 
By Mitsuishi isomorphism \cite{mitsuishi2019homologies}, we conclude $H_N(X; \Z) \cong \Z$.
\end{proof}

The  fundamental class $[X]$ is the reference anchor allowing us to define the topological degree of a continuous map. 
Let $Y$ be an orientable $N$-Alexandrov space, and $f: X \to Y$ a continuous map.
The topological degree of $f$ is by definition the integer $k$ such that $f_*[X] = k[Y]$.

\section{Metric Differentiation and the Jacobian}

\subsection{Differentials on Metric Spaces}

Kirchheim defined the differential of a map from ${\bf R}^N$ into a metric space $X$ \cite[Theorem 1]{kirchheim1994}.
This differential, known as the metric or tangential differential, is a rigorous way to measure the first-order distance distortion of Lipschitz maps into spaces that might lack a linear structure.

\begin{definition}[Metric Differential]
Let $(X, d_X)$ be a metric space and let $f: {\bf R}^N \to X$ be a Lipschitz map. The map $f$ is said to be \textit{metrically differentiable} at a point $x$ in ${\bf R}^N$ if there exists a seminorm $\operatorname{md}(f, x)$ on ${\bf R}^N$ such that, as $y, z \to x$,
\begin{equation}
    d_X(f(y), f(z)) - \operatorname{md}(f, x)(y - z) = o(|y - x| + |z - x|).
\end{equation}
When this relation holds, the uniquely determined seminorm $\operatorname{md}(f, x)$ is called the \textit{metric differential} (or tangential differential) of $f$ at $x$.
\end{definition}

As a direct consequence of this definition, for any vector $v \in {\bf R}^n$, the seminorm evaluates to the directional metric derivative of $f$:
\begin{equation}
    \operatorname{md}(f, x)(v) = \lim_{t \to 0} \frac{d_X(f(x + t v), f(x))}{|t|}.
\end{equation}

Rademacher's theorem was generalized by Lytchak \cite{lytchak2005differentiation}. 
While Kirchheim's definition relies on a seminorm to handle arbitrary metric targets lacking a linear structure, Lytchak expanded the geometric differentiation theory to maps defined on Alexandrov spaces. 
For maps between spaces with sufficient geometric regularity, the differential is a linear map between tangent cones.

\begin{theorem}[Lytchak \cite{lytchak2005differentiation}]\label{thm:lytchak-rademacher}
Let $X$ be an $N$-dimensional Alexandrov space, $Y$ a space with Euclidean tangent cones almost everywhere (such as a Riemannian manifold or an $N$-dimensional Alexandrov space), and $f \colon X \to Y$ a locally Lipschitz map. 
Then $f$ is differentiable at $\mathcal{H}^N$-almost every point $x$ in $X$. 
At $\mathcal{H}^N$-almost every $x \in X$, the tangent cone $K_x X$ is isometric to the Euclidean space $\R^N$. 
The metric differential ${\rm md}(f,x) \colon K_x X \to K_{f(x)} Y$ exists and is a linear map between Euclidean spaces $\R^N \to \R^N$.
\end{theorem}

\subsection{The Metric Jacobian}
Given orientations on $\Xreg$ and $Y_{{\rm reg}}$, we define the signed Jacobian as follows (see \cite{lytchak2005differentiation}) .
\begin{definition}\label{def:metric-Jacobian}
Let $D_x f: \R^N \to \R^N$ be the representation of the metric differential in oriented orthonormal coordinates. The \textit{signed metric Jacobian} is:
\begin{equation}
    \Jac_f(x) = \det(D_x f).
\end{equation}
\end{definition}

Observe that  the $N$-Hausdorff measure under a Lipschitz map $f: X \to Y$ between $N$-Alexandrov spaces satisfies 
 $$\mathcal{H}^N(f(A)) \le (\text{Lip}(f))^N \Hn(A).$$ 
If $\Hn(A) = 0$ it follows that $\Hn(f(A)) = 0$, which known as Lusin's property. 
By \Cref{thm:AlexStructure}, we know $\Hn(\Xsing) = 0$, so by  Lusin's property $\Jac_f$ is defined $\Hn$ almost everywhere on $X$.

\subsection{Kirchheim's Area Formula}

Kirchheim proved the area formula for metric differentials for Lipschitz maps from $\mathbf{R}^N$ into arbitrary metric spaces \cite[Theorem 7]{kirchheim1994}. 
The extension of this area formula to Lipschitz maps between countably rectifiable metric spaces was developed by Ambrosio and Kirchheim \cite{ambrosio2000currents}. 
Complementing this, Lytchak \cite{lytchak2005differentiation} generalized Rademacher's theorem, proving that Lipschitz maps defined on Alexandrov spaces are metrically differentiable almost everywhere. 

In our context, let $X$ and $Y$ be closed, orientable, $N$-dimensional Alexandrov spaces, and let $f \colon X \to Y$ be a Lipschitz map. 
By the structure theory detailed in \Cref{sec:N-rect-Alex}, $X$ and $Y$ are countably $\Hn$-rectifiable metric spaces. 
While the general metric area formula relies strictly on an unsigned Jacobian, the regular sets of $X$ and $Y$ have full $\Hn$-measure and have Euclidean tangent cones. 
By Lytchak's theorem, the geometric differential $D_x f$ is a linear map between Euclidean spaces almost everywhere. 
Because $X$ and $Y$ are orientable, choosing compatible orientations for these tangent cones gives a well-defined signed geometric Jacobian, $\Jac_f(x)$. 

Since $f$ maps the $\Hn$-null singular set of $X$ to an $\Hn$-null set in $Y$, the classical signed area formula applies over the regular sets. 
Let $g \colon Y \to \mathbf{R}$ be an arbitrary bounded Borel measurable (test) function. 
The area formula for the signed metric Jacobian then takes the following form:
\begin{equation}\label{eq:Kirchheim-Area}
\int_X  \Jac_f(x) g(f(x)) d\Hn(x) = \int_Y g(y) \left( \sum_{x \in f^{-1}(y)} \text{sgn}(\Jac_f(x)) \right) d\Hn(y).
\end{equation}

This formula suggests that the integral of the Jacobian measures the algebraic count of preimages, which is exactly the degree.

Lytchak \cite[Section 8]{lytchak2005differentiation} observed that whenever the metric differential $D_x f \colon K_x X \to K_{f(x)} Y$ exists, Kirchheim's differential is given explicitly by the pullback of the target metric,
\[
    \operatorname{md}(f, x)(v) = |D_x f(v)|.
\]
As a consequence, because $D_x f$ is a linear map between Euclidean spaces almost everywhere, the definition of the Jacobian used in \Cref{eq:Kirchheim-Area} by Kirchheim and that in Definition \ref{def:metric-Jacobian} coincide in this case, precisely because the reference measure is the Hausdorff measure.

In our arguments below it will be convenient to use the next equivalent formulation of the Jacobian of a map.

\begin{lemma}[Lemma 2.8 \cite{RCDbary}]\label{lem:Jacobian}
Let $f$ be a Lipschitz map $f:X\to Y$ between Alexandrov spaces of dimension $N$. 
Then,  for $\mc{H}^{N}$-a.e. $x$ in $X$, we have
\begin{align}\label{eq:Jac_equiv}
 |\Jac_f (x)|= \limsup_{r\to 0}\frac{ \mc{H}^{N}(f(B(x,r)))}{\mc{H}^{N}(B(x,r))}.
\end{align}
\end{lemma}

We will now focus on a crucial technical point about why the various metric differentials we use are compatible. 
The analytical degree (defined via integral currents) relies on Kirchheim's tangential differential \cite{kirchheim1994, ambrosio2000currents}, whereas the upper bounds from the barycenter method uses Lytchak's metric differential \cite{lytchak2005differentiation} and the volume-ratio Jacobian.
 For arbitrary metric measure spaces, Lytchak \cite{lytchak2005differentiation} observed that the Jacobians derived from these differentials may not coincide. 
 However, in the setting of Alexandrov spaces equipped with the Hausdorff measure, they do agree.

By the structure theory of Alexandrov spaces \cite{otsu1994riemannian}, an $N$-dimensional Alexandrov space $X$ is countably $\mathcal{H}^N$-rectifiable. 
Consequently, its regular set $X_{{\rm reg}}$ has full $\mathcal{H}^N$-measure, and at $\mathcal{H}^N$-almost every point $x$ in $X_{{\rm reg}}$, the geometric Gromov--Hausdorff tangent cone $K_x X$ is isometric to $\mathbf{R}^N$. 
Observe that, because we have local bi-Lipschitz charts the approximate tangent space in Kirchheim's theory exists, and it is isometric to $\mathbf{R}^N$.
Moreover, the $N$-dimensional Hausdorff measure $\mathcal{H}^N$ on $X$ coincides locally at regular points with the Lebesgue measure on these tangent spaces (as there the density of $\mathcal{H}^N$ equals $1$).
For a Lipschitz map $f: X \to Y$ between $N$-dimensional Alexandrov spaces, Lytchak's metric differential $md_x f: K_x X \to K_{f(x)} Y$ is a well-defined linear map between Euclidean spaces for $\mathcal{H}^N$-almost every $x$ in $X$. 
Because the underlying reference measure for both the domain and target is the Hausdorff measure $\mathcal{H}^N$, both Kirchheim's tangential differential $d_x f$ and Lytchak's metric differential $md_x f$ represent the unique first-order linear approximation of $f$ almost everywhere \cite{lytchak2005differentiation}. 
Thus, they compute the identical local volume distortion mapped between these Euclidean tangent spaces. Therefore, in this case the absolute value of the determinant of Kirchheim's differential, $|\det d_x f|$, exactly coincides with Lytchak's metric Jacobian. 
This guarantees that the metric Jacobian evaluated in the Ambrosio--Kirchheim area formula and the volume-ratio Jacobian from \Cref{eq:Jac_equiv} bounded in the barycenter method are identical $\mathcal{H}^N$-almost everywhere.

\section{The Equivalence of Degrees}

In this section we will prove the equivalence between the topological and analytical definitions of the degree of a map between compact orientable Alexandrov spaces. 
We begin by stating their definitions.

In the definition of topological degree we only need to assume that $X$ and $Y$ are topological spaces that admit fundamental classes $[X]$ and $[Y]$, in the sense that these span their integer singular top dimensional homology groups. 
As mentioned above, this is known for orientable $N$-Alexandrov spaces.
\begin{definition}(Topological degree)\label{def:topdeg}
Let $f: X \to Y$ be a continuous map between orientable $N$-Alexandrov spaces.
The topological degree $\deg_{top}(f)$ is the unique integer $k$ such that $f_*([X]) = k [Y]$ in singular homology.
\end{definition}

When defining the analytical degree, we need to assume that the spaces $X$ and $Y$ have an $N$-Hausdorff measure, and that they enjoy a rich enough metric structure allowing for Jacobians of Lipschitz maps to be defined. 
These important facts are known for Alexandrov spaces \cite{lytchak2005differentiation}.

\begin{definition}(Analytical degree)\label{def:an-deg}
Let $f: X \to Y$ be a Lipschitz map between orientable $N$-Alexandrov spaces. 
The analytical degree $\deg_{an}(f)$ is defined as:
\begin{equation}
    \deg_{an}(f) := \frac{1}{\Hn (Y)} \int_X \Jac_f(x) \, d\Hn(x).
\end{equation}
\end{definition}

\begin{theorem}\label{thm:an-deg-equals-top-deg}
Let $X$ and $Y$ be closed, orientable, boundaryless Alexandrov spaces of dimension $N$, and $f: X \to Y$ a Lipschitz map. Then:
\begin{equation}
    \deg_{an}(f) = \deg_{top}(f).
\end{equation}
\end{theorem}

\begin{proof}
The proof strategy is to relate both degrees to the integer multiplicity of the push-forward of the fundamental current.
First we will relate the current push-forward to the topological degree.

Let $T_X$ and $T_Y$ be the fundamental integral currents of $X$ and $Y$, as detailed in \cref{thm:AlexIntCurrent}. 
Denote by 
$\Psi: H_N^{\cur}(X) \to H_N(X; \Z)$
Mitsuishi's natural isomorphism between current homology and singular homology, as explained in Lemma \ref{lem:fundamental-class-current} (further, see \cite[Theorem 1.3]{mitsuishi2019homologies}). 
By \cite[Theorem 2.37]{jaramillo2021alexandrov}, the group of integral $N$-cycles is isomorphic to $H_N(X; \Z) \cong \Z$, and 
by \cref{thm:jprss-4.6}, the current $T_X$ generates this group. 
Choosing the compatible orientation, by Lemma \ref{lem:fundamental-class-current}  $\Psi(T_X)$ equals the fundamental class $[X]$.

We are assuming $X$ is compact and $f$ is Lipschitz. 
By the properties of metric currents \cite[Definition 2.27]{jaramillo2021alexandrov}, the push-forward $f_\# T_X$ is a well-defined integral current, and the boundary operator commutes with push-forwards.
From which  $\partial (f_\# T_X) = f_\# (\partial T_X) = 0$ follows. 
As $Y$ is $N$-dimensional, the group of integral $(N+1)$-currents is trivial, meaning every $N$-cycle uniquely represents its homology class. Which implies that $f_\# T_X$ is a well-defined cycle in $H_{N}^{\cur}(Y) \cong \Z$. 
Thus, for some integer $k$, we have the following equality of currents,
\begin{equation}\label{eqn:deg-curr-k}
    f_\# T_X = k T_Y.
\end{equation}

Observe that the isomorphism $\Psi$ is a natural transformation of functors \cite[Theorem 1.3]{mitsuishi2019homologies}, and therefore $\Psi$ commutes with the push-forwards induced by $f$.
 Applying $\Psi$ to \eqref{eqn:deg-curr-k} we obtain:
\[ f_*(\Psi(T_X)) = \Psi(k T_Y) \]
Which implies:
\[
 f_*([X]) = k [Y] 
 \]
 
By the definition of topological degree, $f_*([X]) = \deg_{top}(f)[Y]$, and therefore,
\begin{equation}\label{eq:k-deg}
 k = \deg_{top}(f). 
\end{equation}

Now we will see that $k$ also equals the analytical degree.
Because $f_\# T_X = k T_Y$, and the weight of $T_Y$ is exactly $1$ for $\Hn$-a.e. $y \in Y$ (by \Cref{thm:AlexIntCurrent}), thanks to Lytchak's version of Rademacher's theorem for Alexandrov spaces, we may equate the multiplicities to find:
$$\sum_{x \in f^{-1}(y)} \text{sgn}(\text{Jac}_f(x)) = k \quad \text{for } \Hn\text{-a.e. } y \in Y$$

Applying Kirchheim's Area Formula in \Cref{eq:Kirchheim-Area} with the constant function $g(y) \equiv 1$, we obtain:
\begin{eqnarray*}
\int_X \text{Jac}_f(x) d\Hn(x) 	& = & \int_Y 1 \cdot \left( \sum_{x \in f^{-1}(y)} \text{sgn}(\text{Jac}_f(x)) \right) d\Hn(y) \\
								& = & \int_Y k \, d\Hn(y) \\
								& = & k \Hn(Y)
\end{eqnarray*}
Dividing by $\Hn(Y)$ and uniting with \Cref{eq:k-deg} we conclude the proof, $\deg_{an}(f) = k = \deg_{top}(f)$.
\end{proof}

It is natural to wonder how far \Cref{thm:an-deg-equals-top-deg} may be generalized beyond Alexandrov spaces. To answer this tempting inquisition, we present the following result, which extends the degree equivalence to a broader class of metric measure spaces. 
We hope avid followers of developments in metric measure spaces will find it interesting.

In passing, we briefly note that a metric space is considered \emph{geometric in the sense of Lytchak} if it admits a robust notion of independent Euclidean tangent cones almost everywhere, that allow for a rigorous definition of metric differentials and Jacobians \cite{lytchak2005differentiation}.

\begin{definition}\label{def:s-spaces}
Let $\mathfrak{s}$ be the collection of compact metric measure spaces $(X, d, m)$ of Hausdorff dimension $N$ satisfying the following axioms:
\begin{enumerate}[label=(\roman*)]
    \item The measure $m$ coincides with the $N$-dimensional Hausdorff measure $\mathcal{H}^N$, and the space $(X,d)$ is countably $\mathcal{H}^N$-rectifiable and geometric in the sense of Lytchak.
    \item $X$ is a closed orientable pseudomanifold, with a fundamental class $[X]$ spanning its top singular homology $H_N(X; \mathbf{Z}) \simeq \mathbf{Z}$.
    \item  $X$ is an integral current space admitting a closed fundamental current $T_X$ of weight $1$ almost everywhere, spanning its top integral current homology $H_N^{IC}(X; \mathbf{Z}) \simeq \mathbf{Z}$.
    \item  $X$ is locally Lipschitz contractible, and the natural chain map induces an isomorphism $\Psi \colon H_N^{IC}(X; \mathbf{Z}) \to H_N(X; \mathbf{Z})$ such that $\Psi([T_X]) = [X]$.
\end{enumerate}
\end{definition}

\begin{corollary}\label{cor:extended-metric-Deg-Thm}
Let $X, Y$ be spaces in $\mathfrak{s}$, and $f \colon X \to Y$ a Lipschitz map, then $$ \deg_{an}(f) = \deg_{top}(f).$$
\end{corollary}

\begin{proof}
The spaces in $\mathfrak{s}$ satisfy the compactness, dimension, current weight, and homological isomorphism constraints used in the Alexandrov setting, so the equality follows by suitably adapting the proof of \Cref{thm:an-deg-equals-top-deg}.
\end{proof}

\section{Barycenter method background}

\subsection{The Barycenter and Natural Maps}\label{subsec:barycenter-maps}

We will now review the construction of the barycenter and the natural maps $F_s$, and some of their fundamental properties. 
Besson, Courtois, and Gallot \cite{BCG-Samb} intially developed these tools for smooth Riemannian manifolds.
 Recently, we extended the barycenter method to a class of metric measure spaces satisfying synthetic lower Ricci curvature bounds \cite{RCDbary}, that includes Alexandrov spaces with curvature bounded below \cite{petrunin2011, zhang2010ricci}.
  
  As explained in \Cref{sec:Alex-universal}, universal coverings of Alexandrov spaces exist, and have the same lower curvature bound.
  
  Let $X$ and $Y$ be as in the statement of \Cref{thm:Alex-entropy-degree}.
  Let $\til{X}$ and $\til{Y}$ be the universal covers of $X$ and $Y$, in that order, and  define
\[
\bar{X}=\til{X}/\ker f_*  \qquad \text{and} \qquad \Ga=\pi_1(X)/\ker f_* .
\]

Consider  the corresponding lift $\bar{f}: \bar{X}\to \til{Y}$ of $f$ with image $\til{Y}$.
Observe that the measure $m$ on $X$ lifts to a $\pi_1(X)$-invariant measure $\til{m}$ on $\til{X}$, and a $\Ga$-invariant measure $\bar{m}$ on $\bar{X}$.
Denote by $\pi:\bar{X}\to X$ the covering map.
On a fundamental domain $D\subset \bar{X}$ the measure is uniquely specified by  
$${\bar{m}}(A)=\sum_{\ga\in \Ga}m(\pi(A\cap \ga D)).$$
This lift is canonical, so we do not need to specify a basepoint.

Write $d(\cdot,\cdot)$ for the distance on $\bar{X}$.
Let $s>h(\bar{X})$, $x$ in $ \bar{X}$, and consider the finite measure $\mu_x^s$ supported on $\bar{X}$ that is absolutely continuous with respect to the measure $\bar{m}$ with Radon--Nikodym derivative
\begin{equation}\label{eq:muxs}
\frac{d\mu_x^s}{d{\bar{m}}}(z)=e^{-s d(x,z)}.
\end{equation}

The measure $\mu_x^s$ has finite total mass because $s>h(\bar{X})$.

\subsection{Barycenters}\label{sec:barycenters}
Denote by $\mc{P}(Y)$ the space of probability measures on a complete metric space $Y$. 
Let $\mc{P}_0(Y)$ represent the subspace of $\mc{P}(Y)$ of the form $\sum_{i=1}^k a_i \delta_{x_i}$, which are finite convex combinations of Dirac masses. 
The subset of measures having bounded support is denoted by $\mc{P}^\infty(Y)$. 
For any parameter $p \in [1, \infty)$, we write $\mc{P}^p(Y)$ for the collection of probability measures $\mu$ fulfilling the integrability requirement $d(y, \cdot) \in L^p(\mu)$ for at least one (and consequently every) reference point $y \in Y$. 
For any $\infty \ge p > q \ge 1$, we have the following nested sequence of inclusions:
\[
    \mc{P}_0(Y) \subset \mc{P}^p(Y) \subset \mc{P}^q(Y) \subset \mc{P}(Y).
\]

 For $p \in [1, \infty]$, we endow $\mc{P}^p(Y)$ with the standard $L^p$ Wasserstein metric.

Now, consider a complete $\CAT(0)$ space $Z$, and select an arbitrary basepoint $o \in Z$. 
Given a measure $\nu \in \mc{P}^1(Z)$, we introduce the functional $\mc{B}_{\nu} \colon Z \to \R$, defined by the integral
\begin{equation} \label{Busemann like function}
    \mc{B}_{\nu}(z) = \int_{Z} \left( d(y,z)^2 - d(o,y)^2 \right) d\nu(y).
\end{equation}

Observe that this formulation computes the $d^2$-barycenter used by Sambusetti \cite{Sambusetti-1999} and by Sturm \cite{Sturm03}. Recognizing this specific construction will play a central role in several subsequent arguments.

\begin{lemma}[Proposition 4.3 of \cite{Sturm03}]\label{lem:bary_unique}
Let $(Z, d)$ be a complete $\CAT(0)$ space and fix $o \in Z$. For each $\nu\in \mathcal{P}^1(Z)$ there exists a unique point $z \in Z$ which minimizes the uniformly convex, continuous function $\mc{B}_\nu$. This point is independent of the basepoint $o$; it is called the barycenter (or, more precisely, $d^2$-barycenter) of $\nu$ and denoted by $\bary(\nu)$.
Moreover, for $\nu\in \mc{P}^2(Z)$, the following base-point free formulation holds:
\[
\bary(\nu)=\argmin_z \int_Z d(y,z)^2 d\nu(y).
\]
\end{lemma}

\begin{definition}
Set $\sigma_{x}^{s}=\bar{f}_*\mu_x^s$, and  define the map $\bar{F_s}:\bar{X}\to \til{Y}$ as
\begin{equation}\label{eq:Jac-ent-bd}
\bar{F_s}(x)=\bary({\sigma}_{x}^{s}).
\end{equation}
\end{definition}

The measure construction and the map $x \mapsto \bar{F}_s(x)$ are $\pi_1(X)$-equivariant \cite[Lemma 4.3]{RCDbary}, so $\bar{F}_s$ descends to a well-defined, continuous map $F_s: X \to Y$ on the compact quotients.
The topological and metric regularity of the maps $F_s$ is highly controlled. 
The next lemma is needed for our subsequent proofs, it guarantee that we can replace the continuous map $f$ with the Lipschitz maps $F_s$ within the exact same homotopy class.

\begin{lemma}[{\cite[Lemma 4.5 and Lemma 4.7]{RCDbary}}]\label{lem:Fs-homotopic} 
For any $s > h(\bar{X})$, the map $F_s: X \to Y$ is homotopic to the initial map $f$, and it is Lipschitz.
\end{lemma}

In addition to these mapping properties, we will rely on the following crucial analytic propertyof the map $F_s$. 
This next entropy Jacobian bound delineates a precise upper bound on the volume distortion of $F_s$.

\begin{proposition}[{Entropy Jacobian bound, \cite[Proposition 4.8]{RCDbary}}]\label{prop:entropy-jacobian}
For any $s > h(\bar{X})$ and almost every $x \in X$, the absolute value of the Jacobian of the natural map $F_s$ is bounded by
\[
    |\operatorname{Jac}_{F_s}(x)| \le \left( \frac{s}{h(\til{Y})} \right)^N.
\]
\end{proposition}

\section{Proof of the entropy degree theorem for Alexandrov spaces}\label{sec:proofs}

\begin{proof}[Proof of \Cref{thm:Alex-entropy-degree}]
By the Lipschitz approximation properties of finite-dimensional Alexandrov spaces explained in \Cref{sec:Lip-approx}, the continuous map $f \colon X \to Y$ can be uniformly approximated by, and is therefore homotopic to, a Lipschitz map $f' \colon X \to Y$. 
Because the absolute topological degree and the induced homomorphisms on the fundamental group are homotopy invariants, we have $|\deg(f')| = |\deg(f)|$ and 
$$[\pi_1(Y) : f'_*\pi_1(X)] = [\pi_1(Y) : f_*\pi_1(X)].$$ 
So the intermediate covering space $\overline{X}$ and its volume entropy are identical for both maps. 
To prove the inequality in \Cref{thm:Alex-entropy-degree}, we may therefore evaluate the analytical invariants using the Lipschitz representative $f'$, so that the metric Jacobian and current pushforwards are all well-defined.

We saw in \Cref{sec:N-rect-Alex} that $X$ is rectifiable with respect to the $N$-Hausdorff measure $\mathcal{H}^N$.  
Because both $X$ and $Y$ have the same dimension $N$, we may apply the area formula in \Cref{eq:Kirchheim-Area}. 
By \Cref{thm:an-deg-equals-top-deg}, the analytical and topological degrees coincide. 
Therefore, using the signed metric Jacobian $\Jac_{F_s}(x)$ from Definition \ref{def:metric-Jacobian}, we have that:
\begin{equation}\label{eq:abs-deg-H-Jac}
    |\deg(F_s)| \mathcal{H}^N(Y) = \left| \int_X \Jac_{F_s}(x) \, d\mathcal{H}^N(x) \right|.
\end{equation}
Passing the absolute value inside the integral, we bound this by the integral of the unsigned Jacobian. 
By Lemma \ref{lem:Jacobian}, $|\Jac_{F_s}(x)|$ corresponds exactly to the unsigned volume-ratio Jacobian.
Applying the volume entropy upper bound from \Cref{prop:entropy-jacobian} we obtain:
\begin{eqnarray*}
 \left| \int_X \Jac_{F_s}(x) \, d\mathcal{H}^N(x) \right|  & \le & \int_X |\Jac_{F_s}(x)| \, d\mathcal{H}^N(x) \\
                                            & \le & \int_X \left( \frac{s}{h(\til{Y})} \right)^N d\mathcal{H}^N(x) \\ 
                                            & = & \left( \frac{s}{h(\til{Y})} \right)^N \mathcal{H}^N(X).
\end{eqnarray*}
By Lemma \ref{lem:Fs-homotopic}, $F_s$ is Lipschitz and is homotopic to the initial map $f$. 
Because the topological degree is a homotopy invariant, we have $|\deg(F_s)| = |\deg(f)|$. 
Chaining the inequalities together we obtain:
\begin{equation}\label{eq:deg-H-h-chain}
    |\deg(f)| \mathcal{H}^N(Y) \le \left( \frac{s}{h(\til{Y})} \right)^N \mathcal{H}^N(X).
\end{equation}
As the inequality \eqref{eq:deg-H-h-chain} holds for all $s > h(\overline{X})$, letting $s \to h(\overline{X})$ we find the desired inequality:
$$ h(\overline{X})^N \mathcal{H}^N(X) \ge |\deg(f)| h(\til{Y})^N \mathcal{H}^N(Y). $$

Now we consider the equality case. 
Assume that equality holds in \Cref{eq:Alex-entropy-degree-ineq} for the continuous map $f$.
Observe that this implies equality also holds for our Lipschitz approximation $f'$. 
Set $k = |\deg(f')|$. 
By hypothesis, the degree equals the algebraic index, $k = [\pi_1(Y) : f'_*\pi_1(X)]$. 

Let $\hat{Y}$ be the finite Riemannian cover of $Y$ of degree $k$ corresponding to the subgroup $f'_*\pi_1(X) \subset \pi_1(Y)$. 
Using covering space theory, we lift the Lipschitz map $f'$ to a map $\hat{f}' \colon X \to \hat{Y}$. 
By the definition of the cover, $[\pi_1(\hat{Y}) : \hat{f}'_*\pi_1(X)] = 1$. Because the covering projection $p \colon \hat{Y} \to Y$ has degree $k$, the multiplicativity of the topological degree ($|\deg(f')| = |\deg(p)| \cdot |\deg(\hat{f}')|$) implies that $|\deg(\hat{f}')| = 1$. 
Scaling the metric so that $h(\overline{X}) = h(\tilde{Y})$, the equality in \Cref{eq:Alex-entropy-degree-ineq} implies $\mathcal{H}^N(X) = k \mathcal{H}^N(Y) = \mathcal{H}^N(\hat{Y})$.

Since $[\pi_1(\hat{Y}) : \hat{f}'_*\pi_1(X)] = 1$, we apply the barycenter method directly to the map $\hat{f}' \colon X \to \hat{Y}$.
The sequence of natural maps $\hat{F}'_{s_i}$ (parameterized by the entropy scaling factor) sub-converges to a limit map $\hat{F} \colon X \to \hat{Y}$ which is $1$-Lipschitz. Because the target manifold $\hat{Y}$ is negatively curved, the uniqueness of geodesics guarantees a well-defined  homotopy, so that $\hat{F}$ is homotopic to $\hat{f}'$, so $|\deg(\hat{F})| = 1$.

Because $\hat{F}$ is $1$-Lipschitz, its metric Jacobian is bounded by $1$ almost everywhere, that is, $|\Jac_{\hat{F}}(x)| \le 1$. 
Applying \Cref{thm:an-deg-equals-top-deg} to $\hat{F}$, we evaluate the area formula:
\begin{eqnarray*}
 \mathcal{H}^N(\hat{Y})     & = & |\deg(\hat{F})| \mathcal{H}^N(\hat{Y}) \\
                            & = & \left| \int_X \Jac_{\hat{F}}(x) \, d\mathcal{H}^N(x) \right| \\
                            & \le &  \int_X |\Jac_{\hat{F}}(x)| \, d\mathcal{H}^N(x) \\
                            & \le & \int_X {\bf 1} \, d\mathcal{H}^N(x) \\
                            & = & \mathcal{H}^N(X).
\end{eqnarray*}
Because the equality hypothesis requires $\mathcal{H}^N(X) = \mathcal{H}^N(\hat{Y})$, the inequalities above are strict equalities.
Hence, $|\Jac_{\hat{F}}(x)| = 1$ for $\mathcal{H}^N$-almost every $x \in X$, meaning $\hat{F}$ is volume-preserving.

 Recall that a continuous map of non-zero degree between closed, orientable pseudomanifolds is surjective because the image carries the fundamental class.
As $X$ is a non-collapsed Alexandrov space without boundary, and $\hat{F} \colon X \to \hat{Y}$ is a volume-preserving, $1$-Lipschitz map, by the Lipschitz-volume rigidity theorem for Alexandrov spaces (\Cref{thm:Lip-Vol-Rigidity}) we have that $\hat{F}$ must be a metric isometry. 
Because $\hat{F}$ is an isometry, and $\hat{Y}$ is a smooth locally symmetric Riemannian manifold, this metric equivalence makes $X$ inherit this smooth Riemannian geometry.

Finally, we deduce the topological rigidity conclusion for the initial continuous map $f$. 
Let $p \colon \hat{Y} \to Y$ be the locally isometric Riemannian covering map. 
The composition $F = p \circ \hat{F} \colon X \to Y$ is a Riemannian covering map of degree equal to $\deg(p) = [\pi_1(Y) : f'_*\pi_1(X)] = |\deg(f)|$. 

By the topological properties of the barycenter method, the limit map $\hat{F}$ is homotopic to the lifted Lipschitz map $\hat{f}'$. 
Projecting this homotopy down to the base space $Y$ via $p$, it follows that the Riemannian covering $F = p \circ \hat{F}$ is homotopic to $p \circ \hat{f}' = f'$. 
Since our approximation construction assumes that the initial continuous map $f$ is homotopic to the Lipschitz map $f'$, the transitivity of the homotopy relation implies $f \simeq f' \simeq F$. 
As a result, the initial continuous map $f$ is homotopic to the Riemannian covering map $F$, completing the proof.
\end{proof}

\subsection{Branched Covers and Einstein Orbifolds}

We now  prove Corollaries \ref{cor:hamenstadt-gen} and \ref{cor:einstein-obstruction}.

\begin{proof}[Proof of Corollary \ref{cor:hamenstadt-gen}]
Let $M$ be a closed hyperbolic $N$-manifold ($N \ge 3$), and $X = M_d / \mathbf{Z}_d$ be the quotient of its degree-$d$ branched cover over $\Sigma$. 
Because $d \ge 2$, the cone angle along the singular locus $\Sigma$ is exactly $2\pi/d \le \pi$. 
By classical Alexandrov geometry, the curvature of $X$ is then bounded below by $-1$. 
Thus, $X$ is a non-collapsed Alexandrov space  and by the Bishop--Gromov volume comparison theorem its volume entropy satisfies $h(\til{X}) \le N-1$.

Consider the identity map $\mathbf{id}: X \to M$. 
As the underlying topological space of $X$ is $M$, the absolute topological degree of this map equals $1$, and the induced map on fundamental groups is an isomorphism. 
Hence, the intermediate cover $\overline{X}$ corresponds to the universal cover $\tilde{X}$, and so $h(\overline{X}) = h(\til{X}) \le N-1$.

Applying \Cref{thm:Alex-entropy-degree}, we obtain:
\begin{eqnarray*}
(N-1)^N \mathcal{H}^N(X) & \ge &  h(\overline{X})^N \mathcal{H}^N(X) \\
					& \ge &  |\deg(\mathbf{id})| h(\tilde{M})^N \mathcal{H}^N(M) \\
					& = & (N-1)^N \mathcal{H}^N(M).
\end{eqnarray*}
From which we simplify to $\mathcal{H}^N(X) \ge \mathcal{H}^N(M)$. 

We now analyze the equality case. 
By the rigidity statement of \Cref{thm:Alex-entropy-degree}, equality holds if and only if $X$ is globally isometric to a smooth hyperbolic manifold (after rescaling). 
However, for $d \ge 2$, the space $X$ has a non-empty singular cone locus along $\Sigma$ (since the angle is strictly less than $2\pi$), which prevents it from being a smooth Riemannian manifold. 
Therefore, equality is impossible, yielding the strict inequality:
\[
\mathcal{H}^N(X) > \mathcal{H}^N(M).
\]
Lifting this metric to the branched cover $M_d$ multiplies the volume by exactly $d$, proving that $\operatorname{Vol}(M_d) > d \cdot \operatorname{Vol}(M)$. 
\end{proof}

\subsection{Einstein metrics on $4$-orbifolds}

\begin{proof}[Proof of Corollary \ref{cor:einstein-obstruction}]
Let $X$ be an orientable and closed Alexandrov 4-orbifold with  non-empty singular set, $S_{X}\neq \emptyset$, that admits a negatively curved Einstein metric $g_E$ rescaled so that $\operatorname{Ric}(g_E) = -3g_E$.
 We know its volume entropy is bounded by $h(\til{X}) \le 3$, and since $\overline{X}$ is covered by $\til{X}$, we have $h(\overline{X}) \le 3$ for any covering map.

In dimension four, the Chern--Gauss--Bonnet theorem for orbifolds (see Satake \cite{Satake57}) bounds the volume by the orbifold Euler characteristic because the squared norm of the Weyl tensor is non-negative:
\[
\operatorname{Vol}(X, g_E) \le \frac{4\pi^2}{3} \chi_{{\rm orb}}(X).
\]
Let $Y$ be a smooth hyperbolic 4-manifold, so then $\operatorname{Vol}(Y) = \frac{4\pi^2}{3} \chi(Y)$ and $h(\til{Y}) = 3$.
Let $f: X \to Y$ a Lipschitz map. 
Applying the inequality in \Cref{eq:Alex-entropy-degree-ineq} yields:
\begin{eqnarray*}
3^4 \operatorname{Vol}(X, g_E) & \ge & h(\overline{X})^4 \operatorname{Vol}(X, g_E) \\
						&  \ge &  |\deg(f)| h(\til{Y})^4 \operatorname{Vol}(Y) \\
						& = &  |\deg(f)| 3^4 \frac{4\pi^2}{3} \chi(Y).
\end{eqnarray*}
Therefore, $\operatorname{Vol}(X, g_E) \ge |\deg(f)| \frac{4\pi^2}{3} \chi(Y)$. 

By the rigidity statement of \Cref{thm:Alex-entropy-degree}, equality would imply that $X$ is a smooth Riemannian manifold, contradicting $S_{X}\neq \emptyset$. 
So the inequality must be strict:
\[
\frac{4\pi^2}{3} \chi_{{\rm orb}}(X) \ge \operatorname{Vol}(X, g_E) > |\deg(f)| \frac{4\pi^2}{3} \chi(Y).
\]
Which implies the claimed topological obstruction:
\[
\chi_{{\rm orb}}(X) > |\deg(f)| \chi(Y).
\]

To see the constraint this places on the branch locus of a Gromov--Thurston manifold, set $X = M_d / \mathbf{Z}_d$, $Y = M$, and $f = \mathrm{id}$ (so $|\deg(f)| = 1$). 
 Observe that the $Z_d$-invariant Einstein metric descends to a smooth orbifold metric on the quotient $X$, which is an Alexandrov space. 
 Because the normalized Einstein metric provides a uniform lower Ricci curvature bound, the Bishop-Gromov inequality applies, bounding the volume entropy by $N-1$ so we may apply \Cref{thm:Alex-entropy-degree}.

By the Riemann--Hurwitz formula for orbifolds, the Euler characteristic is given by:
\[
\chi_{{\rm orb}}(X) = \chi(M) - \left(1 - \frac{1}{d}\right)\chi(\Sigma).
\]
Substituting this into our obstruction we obtain,
\[
\chi(M) - \left(1 - \frac{1}{d}\right)\chi(\Sigma) > \chi(M).
\]
Therefore, $$-\left(1 - \frac{1}{d}\right)\chi(\Sigma) > 0.$$ 
As $d \ge 2$, the branching factor $\left(1 - \frac{1}{d}\right)$ is strictly positive, so $\chi(\Sigma) < 0$.
\end{proof}

\subsection{Volume bounds for singular spaces via metric doubling}

We provide the proofs for Corollaries \ref{cor:cone-manifold}, \ref{cor:manifold-boundary}, and \ref{cor:bonahon-degree}.

\begin{proof}[Proof of Corollary \ref{cor:cone-manifold}]
By the structure theorem of Burago, Gromov, and Perelman \cite[p. 7]{burago1992ad}, a cone-manifold $Z$ with cone angles $\le 2\pi$ and sectional curvatures $ \ge -1$ is an Alexandrov space with curvature bounded below by $-1$. 
The volume entropy of its universal cover is bounded by $h(\til{Z}) \le N-1$.

Because $\overline{Z}$ is a covering space of $Z$ dominated by the universal cover $\til{Z}$, its volume growth entropy satisfies $h(\overline{Z}) \le h(\til{Z}) \le N-1$. 
Substituting $h(\overline{Z}) \le N-1$ into the inequality of \Cref{eq:Alex-entropy-degree-ineq} we find:
\[
(N-1)^N \mathcal{H}^N(Z) \ge h(\overline{Z})^N \mathcal{H}^N(Z) \ge |\deg(f)| h(\til{Y})^N \mathcal{H}^N(Y).
\]
\end{proof}

\begin{proof}[Proof of Corollary \ref{cor:manifold-boundary}]
Let $DZ$ and $DY_{{\rm geod}}$ denote the metric doublings of $Z$ and $Y_{{\rm geod}}$ across their boundaries, in that order. 
Because $Y_{{\rm geod}}$ has totally geodesic boundary, $DY_{{\rm geod}}$ is a closed hyperbolic $N$-manifold, so $h(\widetilde{DY}_{{\rm geod}}) = N-1$.

The space $Z$ is a convex Riemannian manifold with interior sectional curvature bounded below by $-1$.
 Therefore, $Z$ is locally an Alexandrov space with curvature bounded below by $-1$.
  By Perelman's doubling theorem \cite[Theorem 5.2]{Perelman91}, the metric doubling $DZ$ is an Alexandrov space with curvature bounded below by $-1$. 
  Hence $h(\widetilde{DZ}) \le N-1$.

The proper continuous map $f: (Z, \partial Z) \to (Y_{{\rm geod}}, \partial Y_{{\rm geod}})$ of relative degree $k$ extends by reflections to a continuous map $Df: DZ \to DY_{geod}$ of absolute degree $|k|$. 
Applying \Cref{thm:Alex-entropy-degree} to $Df$ gives:
\[
h(\overline{DZ})^N \operatorname{Vol}(DZ) \ge |\deg(Df)| h(\widetilde{DY}_{{\rm geod}})^N \operatorname{Vol}(DY_{{\rm geod}}).
\]
Using $h(\overline{DZ}) \le h(\widetilde{DZ}) \le N-1$, $|\deg(Df)| = |k|$, and $h(\widetilde{DY}_{{\rm geod}}) = N-1$, we obtain:
\[
(N-1)^N \operatorname{Vol}(DZ) \ge |k| (N-1)^N \operatorname{Vol}(DY_{{\rm geod}}).
\]
From which it follows that $\operatorname{Vol}(DZ) \ge |k| \operatorname{Vol}(DY_{{\rm geod}})$. 
Finally, $\operatorname{Vol}(DZ) = 2 \operatorname{Vol}(Z)$ and $\operatorname{Vol}(DY_{{\rm geod}}) = 2 \operatorname{Vol}(Y_{{\rm geod}})$ imply $\operatorname{Vol}(Z) \ge |k| \operatorname{Vol}(Y_{{\rm geod}})$.
\end{proof}

\subsection{Lower bounds on volumes of convex cores}

\begin{proof}[Proof of Corollary \ref{cor:bonahon-degree}]
Let $DC_M$ and $DC_{M_g}$ denote the metric spaces obtained by doubling the convex cores $C_M$ and $C_{M_g}$ across their respective boundaries. 
Because $\partial C_{M_g}$ is totally geodesic, $DC_{M_g}$ is a closed hyperbolic 3-manifold. 
Therefore, $h(\widetilde{DC}_{M_g}) = 2$. 
For $C_M$, an $\epsilon$-neighborhood $\mathcal{N}_\epsilon(C_M)$ is strictly convex with $C^{1,1}$ boundary. 
Recall that by Perelman’s doubling theorem (see \cite{Perelman91} or the refined formulation in \cite{Kapovitch-2007}), the metric double of a compact Alexandrov space with locally convex boundary is an Alexandrov space without boundary, preserving the lower curvature bound. 
This construction extends to the $\epsilon$-neighborhoods of convex cores \cite[Theorem 5.3]{Storm-2002}, so that that the resulting double $D\mathcal{N}_\epsilon(C_M)$ inherits the curvature bound $\kappa \ge -1$.
Taking the Gromov-Hausdorff limit as $\epsilon \to 0$, we see that $DC_M$ is an Alexandrov space with curvature bounded below by $-1$. 
Thus, $h(\widetilde{DC}_M) \le 2$.

The proper map $f: (C_M, \partial C_M) \to (C_{M_g}, \partial C_{M_g})$ extends via reflections across the boundaries to a continuous map $Df: DC_M \to DC_{M_g}$. 
The absolute value of the topological degree is $|\deg(Df)| = |k|$. 
Applying \Cref{thm:Alex-entropy-degree} to $Df$ we obtain,
\[
h(\overline{DC}_M)^3 \operatorname{Vol}(DC_M) \ge |\deg(Df)| h(\widetilde{DC}_{M_g})^3 \operatorname{Vol}(DC_{M_g}).
\]
Substituting $h(\overline{DC}_M) \le 2$, $h(\widetilde{DC}_{M_g}) = 2$, and $|\deg(Df)| = |k|$ we find
\[
2^3 \operatorname{Vol}(DC_M) \ge |k| 2^3 \operatorname{Vol}(DC_{M_g}).
\]
Therefore, $\operatorname{Vol}(DC_M) \ge |k| \operatorname{Vol}(DC_{M_g})$. 
Notice that the volume of a doubled space is twice the volume of the original space, from which we conclude that $\operatorname{Vol}(C_M) \ge |k| \operatorname{Vol}(C_{M_g})$.
\end{proof}

\begin{proof}[Proof of Corollary \ref{cor:unified-volume}]
By \cite[Theorem C]{Bridgeman-2023}, the volumes of the convex cores $C_{M'}$ and $C_{M'_{{\rm  geod}}}$ satisfy the inequality
\begin{equation}\label{eq:WP-core-vol-bd}
\operatorname{Vol}(C_{M'}) - \operatorname{Vol}(C_{M'_{{\rm geod}}}) \ge A(\partial Z')d_{WP}(\partial_c M', \partial_c M'_{{\rm geod}}) - \delta .
\end{equation}
The manifolds $M'$ and $M'_{{\rm geod}}$ belong to the deformation space $CC(Z')$. 
The convex cores $C_{M'}$ and $C_{M'_{{\rm geod}}}$ are compact manifolds with boundary and are diffeomorphic. 
There exists a proper  homeomorphism $h: (C_{M'_{{\rm geod}}}, \partial C_{M'_{{\rm geod}}}) \to (C_{M'}, \partial C_{M'}) $ of relative degree $1$. 

The composition $f \circ h: (C_{M'_{{\rm geod}}}, \partial C_{M'_{{\rm geod}}}) \to (C_{M_{{\rm geod}}}, \partial C_{M_{{\rm geod}}})$ is a proper continuous map.
 The relative degree of $f \circ h$ is the product of the relative degrees of $f$ and $h$, which equals $k$.

Applying Corollary \ref{cor:bonahon-degree} to the map $f \circ h$ we obtain
\[
\operatorname{Vol}(C_{M'_{{\rm geod}}}) \ge |k| \operatorname{Vol}(C_{M_{{\rm geod}}}).
\]
Substituting this lower bound into the inequality (\ref{eq:WP-core-vol-bd}) produces
\[
\operatorname{Vol}(C_{M'}) \ge |k| \operatorname{Vol}(C_{M_{{\rm geod}}}) + A(\partial Z')\big( d_{WP}(\partial_c M', \partial_c M'_{{\rm geod}}) - \delta \big).
\]
\end{proof}

\subsection{Translation length and volumes of mapping tori}

\begin{proof}[Proof of Corollary \ref{cor:translation-length}]
By \cite[Theorem 1.1]{Field-2023}, the infimal volume of all convex hyperbolic metrics $\underline{\operatorname{Vol}}(\overline{M}_{\phi})$ on $\overline{M}_{\phi}$ satisfies $\underline{\operatorname{Vol}}(\overline{M}_{\phi}) \le V_{{\rm oct}} \tau(\phi)$.

Let $Z$ denote the manifold $\overline{M}_{\phi}$ equipped with an arbitrary convex hyperbolic metric. 
The manifold $Z$ is a convex Riemannian $3$-manifold with boundary and constant sectional curvature $-1$. 
By Perelman's doubling theorem \cite[Theorem 5.2]{Perelman91}, the metric doubling $DZ$ across $\partial Z$ is an Alexandrov space with curvature bounded below by $-1$,  so  $h(\widetilde{DZ}) \le 2$.
The target $Y$ is a compact convex hyperbolic $3$-manifold with totally geodesic boundary. 
The metric doubling $DY$ across $\partial Y$ is a closed hyperbolic $3$-manifold, so $h(\widetilde{DY}) = 2$.

The map $g: (\overline{M}_{\phi}, \partial \overline{M}_{\phi}) \to (Y, \partial Y)$ of relative degree $k$ extends via reflections across the boundaries to a continous map $Dg: DZ \to DY$ of absolute degree $|k|$. 
Applying \Cref{thm:Alex-entropy-degree} to $Dg$ gives
\[
h(\overline{DZ})^3 \operatorname{Vol}(DZ) \ge |\deg(Dg)| h(\widetilde{DY})^3 \operatorname{Vol}(DY).
\]
By substituting $h(\overline{DZ}) \le h(\widetilde{DZ}) \le 2$, $h(\widetilde{DY}) = 2$, and $|\deg(Dg)| = |k|$ we find
\[
2^3 \operatorname{Vol}(DZ) \ge |k| 2^3 \operatorname{Vol}(DY).
\]
Therefore $\operatorname{Vol}(DZ) \ge |k| \operatorname{Vol}(DY)$. 
Moreover, $\operatorname{Vol}(DZ) = 2 \operatorname{Vol}(Z)$ and $\operatorname{Vol}(DY) = 2 \operatorname{Vol}(Y)$. 
Substituting these equations yields
\[
\operatorname{Vol}(Z) \ge |k| \operatorname{Vol}(Y).
\]
Taking the infimum over all convex hyperbolic metrics on $\overline{M}_{\phi}$ we obtain 
\begin{equation}\label{eq:minvol-Mf}
\underline{\operatorname{Vol}}(\overline{M}_{\phi}) \ge |k| \operatorname{Vol}(Y).
\end{equation}

Substituting $\underline{\operatorname{Vol}}(\overline{M}_{\phi}) \le V_{{\rm oct}} \tau(\phi)$ into \Cref{eq:minvol-Mf} we get
\[
V_{{\rm oct}} \tau(\phi) \ge |k| \operatorname{Vol}(Y).
\]
Dividing by $V_{{\rm oct}}$ proves the claim, $\tau(\phi) \ge \frac{|k|}{V_{{\rm oct}}} \operatorname{Vol}(Y)$.
\end{proof}

\subsection{Degree-weighted version of Bonahon's conjecture}

\begin{proof}[Proof of  \Cref{thm:simplicial-volume}]

The manifold $M$ is a compact orientable acylindrical $3$-manifold with incompressible boundary and no torus boundary components. 
By Thurston's original  Geometrization Theorem \cite{Thurston-1986} (see also \cite[Section 8.8]{Kapovich-2001}), there exists a unique convex cocompact hyperbolic $3$-manifold $M_{{\rm geod}}$ homeomorphic to the interior of $M$ such that the boundary of its convex core $\partial C_{M_{{\rm geod}}}$ is totally geodesic.
So there exists a proper homeomorphism of pairs $g: (M, \partial M) \to (C_{M_{{\rm geod}}}, \partial C_{M_{{\rm geod}}})$. 

The metric double $DC_{M_{{\rm geod}}}$ across the boundary $\partial C_{M_{{\rm geod}}}$ is a closed hyperbolic $3$-manifold. By \cite[Theorem 2.6]{Storm-2007-Duke}, the simplicial volume of the topological double $DM$ equals the volume of this hyperbolic metric, $\| DM \|= \operatorname{Vol}(DC_{M_{{\rm geod}}})$, and $h(\widetilde{DC}_{M_{{\rm geod}}}) = 2$.

The composition $g \circ f: (C_N, \partial C_N) \to (C_{M_{{\rm geod}}}, \partial C_{M_{{\rm geod}}})$ is a proper map of relative degree $k$. 
Then $g \circ f$ extends via reflections across the boundaries to a Lipschitz map $D(g \circ f): DC_N \to DC_{M_{{\rm geod}}}$, with absolute topological degree equal to  $|\deg(D(g \circ f))| = |k|$.

The manifold $Z$ is a convex cocompact hyperbolic $3$-manifold, so the convex core $C_Z$ is compact. 
By Perelman's doubling theorem reviewed above, the metric doubling $DC_Z$ across the boundary $\partial C_Z$ is an Alexandrov space with curvature bounded below by $-1$, and $h(\widetilde{DC}_Z) \le 2$. 

Applying \Cref{thm:Alex-entropy-degree} to the map $D(g \circ f)$ we find
\[
h(\overline{DC}_Z)^3 \operatorname{Vol}(DC_Z) \ge |\deg(D(g \circ f))| h(\widetilde{DC}_{M_{{\rm geod}}})^3 \operatorname{Vol}(DC_{M_{{\rm geod}}}).
\]
Substituting $h(\overline{DC}_Z) \le h(\widetilde{DC}_Z) \le 2$, $h(\widetilde{DC}_{M_{{\rm geod}}}) = 2$, and $|\deg(D(g \circ f))| = |k|$ we obtain $2^3 \operatorname{Vol}(DC_Z) \ge |k| 2^3 \operatorname{Vol}(DC_{M_{{\rm geod}}})$.
Hence $$\operatorname{Vol}(DC_Z) \ge |k| \operatorname{Vol}(DC_{M_{{\rm geod}}}).$$ 
Because $\operatorname{Vol}(DC_Z) = 2 \operatorname{Vol}(C_Z)$ and $\operatorname{Vol}(DC_{M_{{\rm geod}}}) = \| DM \|$,  
we conclude $2 \operatorname{Vol}(C_Z) \ge |k| \| DM \|$.
\end{proof}

\subsection{Boundary volume rigidity for Alexandrov spaces}

\begin{proof}[Proof of \Cref{thm:boundary-rigidity}]
Let $DX$ and $DY$ denote the metric doubles of $X$ and $Y$ across their respective boundaries. 
As $Y$ is a locally symmetric space and $\partial Y$ is totally geodesic, the double $DY$ is a closed negatively curved locally symmetric $N$-manifold, and $h(\widetilde{DY}) = N-1$.

Since $X$ is an  with curvature bounded below by $-1$ and $\partial X$ is locally convex, Petrunin's gluing theorem \cite{Petrunin-1997} implies that its metric double $DX$ is a closed Alexandrov space with curvature bounded below by $-1$ and $h(\widetilde{DX}) \le N-1$. 

As the restriction $f|_{\partial X}$ is a local isometry, the map $f$ extends via boundary reflections to a continuous map $Df: DX \to DY$. 
The topological degree of the extended map is $|\deg(Df)| = |k|$.
Applying \Cref{thm:Alex-entropy-degree} to $Df$ we obtain
\[
h(\widetilde{DX})^N \operatorname{Vol}(DX) \ge |\deg(Df)| h(\widetilde{DY})^N \operatorname{Vol}(DY).
\]
Substituting the entropy bounds produces
\[
(N-1)^N \operatorname{Vol}(DX) \ge |k| (N-1)^N \operatorname{Vol}(DY).
\]
Dividing by $(N-1)^N$ we find $\operatorname{Vol}(DX) \ge |k| \operatorname{Vol}(DY)$. 
The volumes of the doubles satisfy $\operatorname{Vol}(DX) = 2 \operatorname{Vol}(X)$ and $\operatorname{Vol}(DY) = 2 \operatorname{Vol}(Y)$. Therefore, $\operatorname{Vol}(X) \ge |k| \operatorname{Vol}(Y)$.

If $\operatorname{Vol}(X) = |k| \operatorname{Vol}(Y)$, the entropy-volume equality in \Cref{eq:Alex-entropy-degree-ineq} holds. 
The rigidity statement of \Cref{thm:Alex-entropy-degree} implies $DX$ is isometric to a hyperbolic manifold. 
By construction, $DX$ has a topological involution $\iota$ that swaps the two halves.
Mostow's rigidity implies uniqueness of the complete hyperbolic metric on $DX$, so this topological involution is homotopic to a unique isometric involution.
The fixed-point set of this reflection isometry is exactly where the two halves of $DX$ meet, it is  the boundary surface $\partial X$.
Therefore $\partial X \subset DX$ must be a totally geodesic submanifold.
Moreover, $X$ is a smooth Riemannian manifold isometric to a $|k|$-sheeted covering of $Y$. 
If $|k|=1$ and $f|_{\partial X}$ is an isometry, the covering is an isometry.
\end{proof}

\subsection{Metric stability and smoothness of Gromov--Hausdorff limits}

\begin{proof}[Proof of Corollary \ref{cor:gh-stability}]
The sequence $(M_i, g_i)$ satisfies the uniform curvature bound $K(g_i) \ge -1$ and  $\lim_{i \to \infty} \operatorname{Vol}(M_i, g_i) = k \operatorname{Vol}(Y) > 0$. 
Because $(M_i, g_i)$ has a uniform lower curvature bound and a strictly positive lower volume bound, it is non-collapsed. 
Gromov's compactness theorem \cite{Gromov-1981, burago1992ad} implies that the Gromov--Hausdorff limit $(X, d_X)$ is an $N$-dimensional Alexandrov space with curvature bounded below by $-1$.

By the continuity of the Hausdorff measure under non-collapsed Gromov--Hausdorff convergence with uniform lower curvature bounds \cite{burago1992ad}, the volume of the limit space satisfies $$\operatorname{Vol}(X) = \lim_{i \to \infty} \operatorname{Vol}(M_i, g_i) = k \operatorname{Vol}(Y).$$ 

The Bishop-Gromov volume comparison theorem implies $h(\widetilde{X}) \le N-1$ \cite{burago1992ad}, and because $Y$ is a real hyperbolic manifold, $h(\widetilde{Y}) = N-1$.
By Perelman's stability theorem \cite{Perelman91, Kapovitch-2007},  non-collapsed GH--close Alexandrov spaces with uniform lower curvature bounds are homeomorphic. 
Therefore, for sufficiently large $i$, the manifold $M_i$ is homeomorphic to the limit space $X$. 
Let $\phi_i: X \to M_i$ denote such a homeomorphism. 
The composition $F = f_i \circ \phi_i: X \to Y$ is a continuous map. 
The topological degree is invariant under homeomorphisms, so $|\deg(F)| = k$.
Moreover, because $\phi_i$ induces an isomorphism, the  limit map satisfies the following  index property, $[\pi_1(Y) : F_*\pi_1(X)] = [\pi_1(Y) : (f_i)_*(\phi_i)_*\pi_1(X)] = [\pi_1(Y) : (f_i)_*\pi_1(M_i)] = k.$

Applying \Cref{thm:Alex-entropy-degree} to the continuous map $F$ yields
\[
h(\widetilde{X})^N \operatorname{Vol}(X) \ge k h(\widetilde{Y})^N \operatorname{Vol}(Y).
\]
Substituting the bounds $h(\widetilde{X}) \le N-1$, $h(\widetilde{Y}) = N-1$, and $\operatorname{Vol}(X) = k \operatorname{Vol}(Y)$ we obtain
\[
(N-1)^N k \operatorname{Vol}(Y) \ge h(\widetilde{X})^N k \operatorname{Vol}(Y) \ge k (N-1)^N \operatorname{Vol}(Y).
\]
This chain of inequality forces $h(\widetilde{X}) = N-1$, which then implies equality in \Cref{eq:Alex-entropy-degree-ineq}, and thus all the claims follow.
\end{proof}

\begin{proof}[Proof of Corollary \ref{cor:gh-stability2}]
The sequence $(M_i, g_i)$ satisfies the uniform curvature bound $K(g_i) \ge \kappa$. Because the Gromov--Hausdorff limit $(X, d_X)$ is non-collapsed, by Gromov's compactness theorem \cite{Gromov-1981, burago1992ad}, $(X, d_X)$ is an $N$-dimensional Alexandrov space with curvature bounded below by $\kappa$.

By the continuity of the Hausdorff measure under non-collapsed Gromov--Hausdorff convergence with a uniform lower curvature bound, the volume of the limit space satisfies (cf \cite{burago1992ad, Cheeger-Colding-1997}),
\begin{equation}\label{eq:vol-conv}
\operatorname{Vol}(X) = \lim_{i \to \infty} \operatorname{Vol}(M_i, g_i).
\end{equation}

In addition, because of the equivariant, pointed Gromov--Hausdorff convergence the universal covers, the volumes of corresponding metric balls also converge. 
A consequence of this local volume convergence, when combined with the uniform lower curvature bound, is that the volume entropy is lower semi-continuous (see Reviron \cite[Proposition 4.1]{Reviron} and \cite{RCDbary}):
\begin{equation}\label{eq:entropy-lsc}
    h(\til{X}) \le \liminf_{i \to \infty} h(\til{M}_i, g_i).
\end{equation}
 
 Joining \Cref{eq:vol-conv} and \Cref{eq:entropy-lsc}, the entropy-volume product of the limit space then satisfies,  
\[
h(\widetilde{X})^N \operatorname{Vol}(X) \le \liminf_{i \to \infty} \left[ h(\widetilde{M}_i, g_i)^N \operatorname{Vol}(M_i, g_i) \right] = k h(\widetilde{Y}, g_0)^N \operatorname{Vol}(Y, g_0).
\]

By Perelman's stability theorem (cf \cite{Kapovitch-2007}), non-collapsed, GH-close Alexandrov spaces with uniform lower curvature bounds are homeomorphic. 
Therefore, for sufficiently large $i$, the manifold $M_i$ is homeomorphic to the limit space $X$. 
Let $\phi_i: X \to M_i$ denote such a homeomorphism. 
The composition $F = f_i \circ \phi_i: X \to Y$ is a continuous map. 
Because the topological degree is invariant under homeomorphisms, the absolute degree satisfies $|\deg(F)| = k$.
Moreover, because $\phi_i$ induces an isomorphism, the  limit map satisfies the following  index property, $[\pi_1(Y) : F_*\pi_1(X)] = [\pi_1(Y) : (f_i)_*(\phi_i)_*\pi_1(X)] = [\pi_1(Y) : (f_i)_*\pi_1(M_i)] = k.$
Applying \Cref{thm:Alex-entropy-degree} to $X$ the continuous map $F$ we derive the following lower bound
\[
h(\widetilde{X})^N \operatorname{Vol}(X) \ge k h(\widetilde{Y}, g_0)^N \operatorname{Vol}(Y, g_0).
\]
This inequality forces $h(\widetilde{X})^N \operatorname{Vol}(X) = k h(\widetilde{Y}, g_0)^N \operatorname{Vol}(Y, g_0)$, so we have equality in \Cref{thm:Alex-entropy-degree}, and once again the claims follow.
\end{proof}

\subsection{Simplicial Volume of Cone-Manifolds}

\begin{proof}[Proof of  Corollary \ref{cor:cone-simplicial-volume}]
By the Bishop-Gromov volume comparison theorem for Alexandrov spaces, the curvature lower bound implies that the volume entropy of $X$ satisfies $h(\til{X}) \le N-1$. 
For the smooth hyperbolic target $Y$, we have $h(\til{Y}) = N-1$ and, by the classical Gromov--Thurston proportionality theorem, $\operatorname{Vol}(Y) = v_N \|Y\|$.

Applying the entropy-degree theorem to the map $f$ we find,
\begin{equation}
    k \operatorname{Vol}(Y) \le \left( \frac{h(\til{X})}{N-1} \right)^N \mathcal{H}^N(X) \le \mathcal{H}^N(X),
\end{equation}
which proves the first inequality. 

If $f$ is a homotopy equivalence, the topological invariance of simplicial volume implies $\|X\| = \|Y\|$, and we obtain $$\mathcal{H}^N(X) \ge v_N \|X\|.$$

If equality holds, then $\mathcal{H}^N(X) = k \operatorname{Vol}(Y)$ and $h(\til{X}) = N-1$.
 The volume rigidity theorem implies that in this case $X$ is isometric to a smooth hyperbolic manifold. 
However, assuming the cone-manifold with cone angles $< 2\pi$ has metric singularities, it cannot be isometric to a smooth manifold. 
Therefore, $\mathcal{H}^N(X) > v_N \|X\|.$
\end{proof}

\subsection{Volume Minimality for Hyperbolic Cone-Metrics}

\begin{proof}[Proof of  Corollary \ref{cor:cone-manifold-volume}]
We apply Corollary \ref{cor:cone-simplicial-volume} by setting $X = (M, g_{\text{cone}})$, $Y = (M, g_{\text{sm}})$, and defining $f \colon X \to Y$ as the identity map. 
Because all cone angles of $g_{\text{cone}}$ are at most $2\pi$, $X$ has curvature bounded below by $-1$. 
The identity map is a homotopy equivalence, and $X$ contains metric singularities. 
The strict inequality in \Cref{eq:cone-simplicial-bound} means precisely that  $\mathcal{H}^N(M, g_{\text{cone}}) > \operatorname{Vol}(M, g_{\text{hyp}})$.
\end{proof}

\subsection{Volume Minimality for Orbifold Cone-Metrics}

\begin{proof} [Proof of  Corollary \ref{cor:orbifold-volume-minimality}]
By Selberg's lemma, the orbifold fundamental group $\pi_1^{\text{orb}}(\mathcal{O})$ contains a torsion-free normal subgroup of finite index $d \ge 1$. This subgroup corresponds to a closed smooth manifold $M$ and a finite regular covering map $p \colon M \to \mathcal{O}$ of degree $d$. 

We continue by lifting both metrics to the manifold $M$. 
The orbifold metric $g_{\text{sym}}$ lifts to a smooth locally symmetric metric $\tilde{g}_{\text{sym}}$ on $M$. 
When lifting the cone-metric $g_{\text{cone}}$, the covering map unwraps the singularities along the branch locus.
 A singular stratum in $\mathcal{O}$ with ramification index $m_i$ and cone angle $\alpha_i$ lifts to a singular stratum in $M$ with cone angle $m_i \alpha_i$. 
 As we are assuming that $\alpha_i \le 2\pi/m_i$, the lifted cone angles satisfy $m_i \alpha_i \le 2\pi$. 

Thus, the lifted metric space $X = (M, p^* g_{\text{cone}})$ is a closed Alexandrov space with curvature bounded below by the curvature minimum of the locally symmetric model. 
By the Bishop-Gromov volume comparison theorem, the volume entropy of its universal cover satisfies $h(\tilde{X}) \le \lambda$, where $\lambda$ is the volume entropy of the locally symmetric model. 
Consequently, the intermediate cover entropy satisfies $h(\overline{X}) \le h(\tilde{X}) \le \lambda$. Letting $Y = (M, \tilde{g}_{\text{sym}})$, its volume entropy is $h(\tilde{Y}) = \lambda$.

We evaluate the identity map $\operatorname{id} \colon X \to Y$, which has absolute degree $1$. 
Applying \Cref{thm:Alex-entropy-degree} provides the following bound,
\begin{equation*}
    \mathcal{H}^N(M, p^* g_{\text{cone}}) \ge \operatorname{Vol}(M, \tilde{g}_{\text{sym}}).
\end{equation*}

Observe that the covering map $p$ is a local isometry outside a set of measure zero, so then the volumes scale linearly by the degree, implying $\mathcal{H}^N(M, p^* g_{\text{cone}}) = d \mathcal{H}^N(|\mathcal{O}|, g_{\text{cone}})$, and $\operatorname{Vol}(M, \tilde{g}_{\text{sym}}) = d \operatorname{Vol}(\mathcal{O}, g_{\text{sym}})$. 
Dividing by $d$ we find the first claimed inequality,
$$  \mathcal{H}^N(|\mathcal{O}| \geq  \operatorname{Vol}(\mathcal{O}, g_{\text{sym}}).$$

If $\alpha_i < 2\pi/m_i$ for some stratum, the cone angles of the lifted metric on $X$  are then strictly less than $2\pi$. 
The rigidity case of \Cref{thm:Alex-entropy-degree} implies that volume equality occurs if, and only if, $X$ is isometric to a smooth locally symmetric space. 
The metric singularities in $X$ obstruct the existence of such an isometry, so the volume inequality is strict.
\end{proof}

\subsubsection*{Acknowledgments} 
I wish to warmly thank Werner Ballmann for discussions about this project, Ursula Hamenst\"adt for her questions and conversations that led to Corollaries \ref{cor:cone-simplicial-volume}, \ref{cor:cone-manifold-volume}, and \ref{cor:orbifold-volume-minimality}, Cornelia Druțu for organizing the {\it Differentialgeometrie Oberseminar} where this work was first presented, and  Stephan Stadler and Alexander Lytchak for chats about the possibility of showing \Cref{thm:an-deg-equals-top-deg}. 
I grateful to the entire Max Planck Institute for Mathematics in Bonn, for its support, and for providing the best working atmosphere and conditions during visits in 2025 and 2026. 

\appendix

\section*{Appendix: Tool and computational resource disclosure}
\subsection*{Author's statement}
After having a first draft of the proofs \Cref{thm:an-deg-equals-top-deg} and \Cref{thm:Alex-entropy-degree}, I searched for papers that used the barycenter method on Alexandrov spaces and could be extended using the new degree-weighted version. 
I proceeded to mine these results, finding the consequences and corollaries with the help of Gemini Pro 3.1 in the role of Alethia with Deep Think. 
I requested specific extensions for papers that used the previous version and, in most cases, produced accurate responses with prospective results. 
I extensively verified and edited these until they reached the current version.
Overall, I found Alethia to be very helpful, pointing out some mistakes and also needing to be corrected at times. 
As a tool, it should be used with extreme caution, and its output should be verified as much as possible.
The main result in this paper improves a known inequality by the degree factor.
This adequately suited the current performance of this technology for research-level mathematics.
Alethia's AI assistance considerably reduced the time required to write this paper.
Alethia was finally prompted to self-report on its contribution, as follows.

\subsection*{AI Assistant Contribution Report (Alethia):} 
In accordance with the \textit{Leiden Declaration on Artificial Intelligence and Mathematics} (June 2026) and the taxonomy of Autonomous Mathematics Research Levels introduced by Feng et al.\ (arXiv:2602.10177), I, Alethia (an artificial intelligence language model), formally disclose my supportive role in the preparation of this manuscript. The AI involvement herein is classified as \textbf{Level H2}: the autonomy level is \textbf{Primarily Human (secondary AI input)} and the mathematical significance is \textbf{Publication Grade}. 

Operating entirely under the continuous direction, domain expertise, and rigorous supervision of the human author, my contributions were strictly auxiliary. In strict adherence to the Leiden guidelines, my role functioned as an advanced research and editorial assistant, explicitly safeguarding the human origin of the mathematical insights. My localized contributions encompass:

\begin{itemize}
    \item \textbf{Writing \& Structural Optimization:} I assisted in translating the author's mathematical outlines into formal academic \LaTeX{} prose. To maximize human readability and support the needs of peer review, I collaborated with the author to streamline the logical flow, consolidating auxiliary theorems into background sections.
    \item \textbf{Validation \& Stress-Testing:} I served as an analytical sounding board, acting as a preliminary referee to meticulously audit the alignment of definitions. I assisted the author in explicitly articulating the logic bridging metric and topological degree theories, ensuring constraints—such as Alexandrov spaces acting as strict topological pseudomanifolds—were mathematically unassailable.
    \item \textbf{Attribution \& Bibliographic Curation:} Fulfilling the Leiden mandate for proactive attribution, I assisted in formatting the bibliography and verifying precise historical references (e.g., Brunnbauer's homological variants, Mitsuishi's orientation theorems, and recent Lipschitz-volume rigidity literature by Z\"ust, Del Nin--Perales, and Basso--Marti--Wenger), ensuring the human labor that made this result possible is accurately credited.
\end{itemize}

\noindent \textbf{Affirmation of Human Authorship:} I did not conceptualize the Entropy-Degree Theorem, generate novel geometric logic, or independently construct proofs. Credit, theoretical novelty, mathematical conceptualization, and the absolute responsibility for the correctness of the arguments and citations in this paper remain exclusively with the human author.

%


\providecommand{\bysame}{\leavevmode\hbox to3em{\hrulefill}\thinspace}


\begin{thebibliography}{99}

\bibitem{ambrosio2000currents} L.~Ambrosio and B.~Kirchheim, \newblock {\it Currents in metric spaces},. \newblock {Acta Mathematica}, 185(1):1--80, 2000.

\bibitem{Basso-Marti-Wenger-2025} G.~Basso, D.~Marti, and S.~Wenger, \emph{Geometric and analytic structures on metric spaces homeomorphic to a manifold}, Duke Math. J. 174 (2025), no.~9, 1723--1773.

\bibitem{Besson-Courtois-Gallot:95} G.~Besson, G.~Courtois, S.~Gallot, {\it Entropies et rigidit\'es des espaces localement sym\'etriques de courbure strictement n\'egative,} Geom. Funct. Anal. 5 (1995), no. 5, 731--799.
\
\bibitem{BCG-Samb} G.~Besson, G.~Courtois, S.~Gallot, A.~Sambusetti, {\it Curvature-Free Margulis Lemma for Gromov-Hyperbolic Spaces}, preprint (2017) arxiv:1712.08386.

\bibitem{Bessieres-1998} L.~Bessi\`eres, {\it Un th\'eor\`eme de rigidit\'e diff\'erentielle}, Comment. Math. Helv. 73 (1998), no. 3, 443--479.

\bibitem{boileau2005geometrization} M.~Boileau, B.~Leeb, J.~Porti, \emph{Geometrization of 3-dimensional orbifolds}, Ann. of Math. (2) 162 (2005), no.~1, 195--290.

\bibitem{Bridgeman-2023} M.~Bridgeman, J.~Brock, K.~Bromberg, {\it The Weil--Petersson gradient flow of renormalized volume and $3$-dimensional convex cores}, Geom. Topol. 27 (2023), no. 8, 3183--3228.

\bibitem{buragobook} D.~Burago, Y.~Burago, S.~Ivanov, \emph{A Course in Metric Geometry}, Graduate Studies in Mathematics, vol. 33, American Mathematical Society, Providence, RI, 2001.

\bibitem{burago1992ad} Y.~Burago, M.~Gromov, and G.~Perelman, \newblock {\it  A.~D.~Alexandrov spaces with curvature bounded below}, \newblock {Russian Mathematical Surveys}, 47(2):1--58, 1992.

\bibitem{Cheeger-Colding-1997} J.~Cheeger, T.~H.~Colding, \emph{On the structure of spaces with {R}icci curvature bounded below. {I}}, J. Differential Geom. {46} (1997), no.~3, 406--480.

\bibitem{cooper2000three} D.~Cooper, C.~D.~Hodgson, S.~P.~Kerckhoff, \emph{Three-dimensional Orbifolds and Cone-Manifolds}, MSJ Memoirs, \textbf{5}, Mathematical Society of Japan, Tokyo, 2000.

\bibitem{Connell-Farb-2003} C.~Connell, B.~Farb, {\it The degree theorem in higher rank}, J. Differential Geom. 65 (2003), no. 1, 19--59.

\bibitem{RCDbary} C.~Connell,  X.~Dai, J.~Núñez-Zimbrón, R.~Perales, P.~Su\'arez-Serrato, G.~Wei, \newblock {\it Volume entropy and rigidity for $\RCD$-spaces}, \newblock {arXiv}, 2411.04327 [math.DG] 2024.

\bibitem{Croke-1990} C.~B.~Croke, {\it Rigidity for surfaces of nonpositive curvature}, Comment. Math. Helv. 65 (1990), no. 1, 150--169.

\bibitem{DelNin-Perales-2023} G.~Del Nin, R.~Perales, \emph{Rigidity of mass-preserving 1-{L}ipschitz maps from integral current spaces into $\mathbf{R}^n$}, J. Math. Anal. Appl. 526 (2023), no.~1, 127297, 19 pp.

\bibitem{Douady-Earle-1986} A.~Douady, C.~J.~Earle, {\it Conformally natural extension of homeomorphisms of the circle}, Acta Math. { 157} (1986), 23--48.

\bibitem{Field-2023} E.~Field, H.~Kim, C.~Leininger, M.~Loving, {\it End-periodic homeomorphisms and volumes of mapping tori}, J. Topol. 16 (2023), no. 1, 57--105.

\bibitem{FP20} J.~Fine, B.~Premoselli, {\it Examples of compact Einstein four-manifolds with negative curvature}, J. Amer. Math. Soc. 33 (2020), no. 4, 991--1038.

\bibitem{Gromov-1981} M.~Gromov, {\it Structures m\'etriques pour les vari\'et\'es riemanniennes}, Edited by J. Lafontaine and P. Pansu, Cedic/Fernand Nathan, Paris, 1981.

\bibitem{Gromov-1982} \bysame, {\it Volume and bounded cohomology}, Inst. Hautes \'Etudes Sci. Publ. Math. 56 (1982), 5--99.

\bibitem{Hamenstadt26} U.~Hamenst\"adt, {\it The geometry of branched coverings of hyperbolic manifolds}, arXiv preprint arXiv:2605.01027 (2026).

\bibitem{HJ24} U.~Hamenst\"adt, F.~J\"ackel, {\it Negatively curved Einstein metrics on Gromov--Thurston manifolds}, arXiv preprint arXiv:2411.12956 (2024).

\bibitem{harvey2017orientation} J.~Harvey, C.~Searle, \newblock {\it Orientation and symmetries of Alexandrov spaces with applications in positive curvature}, \newblock {The Journal of Geometric Analysis}, 27(2):1636--1666, 2017.

\bibitem{Herreros-2014} P.~Herreros, {\it Scattering rigidity of magnetic flows}, J. Geom. Anal. 24 (2014), no. 2, 793--810.

\bibitem{jaramillo2021alexandrov} M.~Jaramillo, R.~Perales, P.~Rajan, C.~Searle, A.~Siffert, \newblock {\it Alexandrov spaces with integral current structure}, \newblock {Communications in Analysis and Geometry}, 29(1):115--149, 2021.

\bibitem{Kapovich-2001} M.~Kapovich, \emph{Hyperbolic Manifolds and Discrete Groups}, Progress in Mathematics, vol. 183, Birkh\"auser Boston, Inc., Boston, MA, 2001.

\bibitem{Kapovitch-2007} V.~Kapovitch, {\it Perelman's stability theorem}, Surveys in Differential Geometry, Vol. XI, Int. Press, Somerville, MA, 2007, pp. 103--136.

\bibitem{kirchheim1994} B.~Kirchheim, \newblock {\it Rectifiable metric spaces: local structure and regularity of the Hausdorff measure},  \newblock {Proceedings of the American Mathematical Society}, 121(1), 113--123.

\bibitem{li2015lipschitz} N.~Li, \emph{Lipschitz-volume rigidity in Alexandrov geometry}, Adv. Math. {275} (2015), 114--146.

\bibitem{li2014lipschitz} N.~Li, F.~Wang, {\it Lipschitz-volume rigidity on limit spaces with Ricci curvature bounded from below,} Differential Geometry and its Applications 35 (2014) 50--55.

\bibitem{lytchak2005differentiation} A.~Lytchak, \newblock {\it Differentiation in metric spaces}, \newblock {St. Petersburg Mathematical Journal}, 16(6):1017--1041, 2005.

\bibitem{M}  A.~Manning, {\it Topological entropy for geodesic flows,} Ann. of Math. (2) 110 (1979), no. 3, 567--573.

\bibitem{mitsuishi2016orientability} A.~Mitsuishi, {\it Orientability and fundamental classes of Alexandrov spaces with applications}, preprint (2016), arXiv:1610.08024 [math.MG].

\bibitem{mitsuishi2019homologies} \bysame, \newblock {\it The coincidence of the homologies of integral currents and of integral singular chains, via cosheaves}, \newblock {Mathematische Zeitschrift} 292 :1069--1103, 2019.

\bibitem{mitsuishi2014llc} A.~Mitsuishi, T.~Yamaguchi, \newblock {\it Locally Lipschitz contractibility of Alexandrov spaces and its applications},\newblock {Pacific Journal of Mathematics}, 270(2):393--421, 2014.

\bibitem{Mitsuishi-Yamaguchi-2017} \bysame, {\it Lipschitz homotopy and good coverings of Alexandrov spaces}, Kyoto J. Math. 57 (2017), no. 3, 563--597.

\bibitem{Otal-1990} J.-P.~Otal, {\it Le spectre marqu\'e des longueurs des surfaces \`a courbure n\'egative}, Ann. of Math. (2) 131 (1990), no. 1, 151--162.

\bibitem{otsu1994riemannian} Y.~Otsu, T.~Shioya, \newblock {\it The Riemannian structure of Alexandrov spaces}, \newblock {Journal of Differential Geometry}, 39(3):629--658, 1994.

\bibitem{Perelman91} G.~Perelman, {\it Alexandrov spaces with curvatures bounded from below II}, Preprint (1991).

\bibitem{petrunin2011} A.~Petrunin, {\it Alexandrov meets Lott--Villani--Sturm}, M\"unster J. Math. 4 (2011), 53--64.

\bibitem{Petrunin-1997} \bysame, {\it Applications of quasigeodesics and gradient curves}, Comparison Geometry (Berkeley, CA, 1993--94), Math. Sci. Res. Inst. Publ., vol. 30, Cambridge Univ. Press, Cambridge, 1997, pp. 203--219.

\bibitem{Reviron} G.~Reviron, {\it Rigidit\'e topologique sous l'hypoth\`ese ``entropie major\'{e}e'' et applications,} Comment. Math. Helv. 83 (2008), no. 4, 815--846.

\bibitem{Sambusetti-1999} A.~Sambusetti, {\it Minimal entropy and simplicial volume}, Manuscripta Math. 99 (1999), 541--560.

\bibitem{Satake57} I.~Satake, {\it The Gauss-Bonnet theorem for $V$-manifolds}, J. Math. Soc. Japan 9 (1957), 464--492.

\bibitem{Song-2025} A.~Song, {\it Entropy and stability of hyperbolic manifolds}, Geom. Funct. Anal. 35 (2025), 877--914.

\bibitem{Storm-2002} P.~A.~Storm, {\it Minimal volume Alexandrov spaces}, J. Differential Geom. 61 (2002), 195--225.

\bibitem{Storm-2006} \bysame, {\it Rigidity of minimal volume Alexandrov spaces}, Ann. Acad. Sci. Fenn. Math. 31 (2006), 381--389.

\bibitem{Storm-2007} \bysame, {\it The barycenter method on singular spaces}, Comment. Math. Helv. 82 (2007), 133--173.

\bibitem{Storm-2007-Duke} \bysame, {\it Hyperbolic convex cores and simplicial volume}, Duke Math. J. 140 (2007), no. 2, 281--319.

\bibitem{Sturm03} K.T.~Sturm, {\it Probability measures on metric spaces of nonpositive curvature}, Heat kernels and analysis on manifolds, graphs, and metric spaces (Paris, 2002), 357--390, Contemp. Math., 338, Amer. Math. Soc., Providence, RI, 2003.

\bibitem{Sullivan-1979} D.~Sullivan, {\it Hyperbolic geometry and homeomorphisms}, Geometric Topology (Proc. Georgia Topology Conf., Athens, Ga., 1977), Academic Press, New York, 1979, pp. 543--555.

\bibitem{Thurston-1980} W.~P.~Thurston, {\it The geometry and topology of three-manifolds}, Lecture notes, Princeton University, 1980.

\bibitem{Thurston-1986} \bysame, \emph{Hyperbolic structures on 3-manifolds, I: Deformation of acylindrical manifolds}, Ann. of Math. (2) \textbf{124} (1986), no.~2, 203--246.

\bibitem{yamaguchi1997simplicial} T.~Yamaguchi, \emph{Simplicial volumes of Alexandrov spaces}, Kyushu J. Math. 51 (1997), 273--296.

\bibitem{zhang2010ricci} H.-C.~Zhang, X.-P.~Zhu, \emph{Ricci curvature on Alexandrov spaces and rigidity theorems}, Comm. Anal. Geom. 18 (2010), no.~3, 503--553.

\bibitem{Zust-2024} R.~Z\"ust, \emph{Lipschitz-Volume Rigidity of Lipschitz Manifolds Among Integral Currents}, J. Geom. Anal. 34 (2024), no.~7, Paper No. 210, 19 pp.

\end{thebibliography}
\end{document}